\documentclass[12pt]{amsart}
\usepackage{amssymb,amsmath,tabularx,mathrsfs,mathbbol,yfonts,upgreek}
\usepackage{amsthm,verbatim,comment}
\usepackage[bookmarks=true]{hyperref}
\usepackage{pstricks,pst-node,pst-plot}
\usepackage{geometry}
\usepackage{stmaryrd}
\usepackage{paralist}
\usepackage[all]{xy}
\usepackage{array}
\usepackage{bm}
\numberwithin{equation}{section}

\DeclareSymbolFontAlphabet{\mathbb}{AMSb}
\DeclareSymbolFontAlphabet{\mathbbl}{bbold}

\newtheorem{thm}{Theorem}[section]
\newtheorem{thmx}{Theorem}

\newtheorem{lem}[thm]{Lemma}
\newtheorem{prop}[thm]{Proposition}
\newtheorem{cor}[thm]{Corollary}

\theoremstyle{definition}

\newtheorem{nota}[thm]{Notation}

\newtheorem{rem}[thm]{Remark}
\theoremstyle{remarks}

\newtheorem*{rem*}{Remarks}

\newtheoremstyle{case}{}{}{}{}{}{:}{ }{}
\theoremstyle{case}

\theoremstyle{plain}
\newtheorem*{conj*}{Conjecture}

\newcommand{\Tr}{\mathrm{Tr}}

\newcommand{\sym}[1]{\mathfrak{S}_{#1}}

\renewcommand{\O}{\mathcal{O}}

\newcommand{\N}{\mathrm{N}}

\newcommand{\B}{\mathcal{B}}

\newcommand{\F}{\mathbb{F}}
\newcommand{\K}{\mathbb{K}}

\title[On some trivial source Specht modules]{On some trivial source Specht modules}

\begin{document}
\author{Yu Jiang}
\address[Y. Jiang]{Division of Mathematical Sciences, Nanyang Technological University, SPMS-MAS-05-34, 21 Nanyang Link, Singapore 637371.}
\email[Y. Jiang]{jian0089@e.ntu.edu.sg}


\begin{abstract}
The paper presented here focuses on the classification of trivial source Specht modules. We completely classify the trivial source Specht modules labelled by hook partitions. We also classify the trivial source Specht modules labelled by two-part partitions in the odd characteristic case. Moreover, in the
even characteristic case, we prove a result for the classification of the trivial source Specht modules labelled by partitions with $2$-weight $2$, which justifies a conjecture of \cite{Tara}.
\end{abstract}
\maketitle
\noindent \textbf{Keywords.} Trivial source module; Symmetric group; Specht module\ \ \textbf{MSC.}  20C30
\section{Introduction}
Let $G$ be a finite group and $\F$ be an algebraically closed field of positive characteristic $p$. The trivial source $\F G$-modules, defined as indecomposable $\F G$-modules with trivial sources, are important objects in the modular representation theory of $\F G$. Though the trivial source $\F G$-modules are known to have plenties of properties, it is usually very difficult to classify explicitly all the trivial source $\F G$-modules. Even in the context of group algebras of symmetric groups, the question of classifying all the trivial source modules is still poorly understood.

For the modules of symmetric groups, it is already well-known that the signed Young modules are trivial source modules (see \cite[Theorem 3.8 (b)]{DanzLim}). However, little knowledge is known in the classification of trivial source Specht modules. Let $\sym{n}$ be the symmetric group on $n$ letters. According to our knowledge, the project of classifying all trivial source Specht modules of $\F\sym{n}$ first begins in the thesis of Hudson (see \cite{Tara}). In her thesis, Hudson classified all the trivial source Specht modules labelled by partitions with $p$-weight $1$ when $p>3$ and identified the trivial source Specht modules of $\F\sym{2p}$ and $\F \sym{2p+1}$ when $p$ is odd. Moreover, for the case $p=2$, she provided the following conjecture.
\begin{conj*}\cite[Conjecture 5.0.1]{Tara}
Let $p=2$. Then the trivial source Specht modules labelled by partitions with $2$-weight $2$ are either simple or isomorphic to Young modules.
\end{conj*}
The paper presented here pursues the topic of the project. The main approaches used in the paper depend on the fixed-point functors of symmetric groups defined in \cite{Hemmer2} and the  Brou\'{e} correspondence of trivial source $\F\sym{n}$-modules. We completely classify the trivial source Specht modules labelled by hook partitions. When $p>2$, we also classify the trivial source Specht modules labelled by two-part partitions. Moreover, for $p=2$ we prove the conjecture of Hudson. To state our results more precisely, let $\lambda$ be a partition of $n$ and denote by $S^\lambda$ and $Y^\lambda$ the Specht module and Young module labelled by $\lambda$ respectively. Take $JM(n)_p$ to be the set of partitions labelling simple Specht modules of $\F \sym{n}$. We will show the following:

\begin{thmx}\label{T;A}
Let $n$ and $r$ be integers such that $0\leq r<n$. Then $S^{(n-r,1^r)}$ is a trivial source module if and only if $(n-r, 1^r)\in JM(n)_p$ or one of the following holds:
\begin{enumerate}
\item [\em(i)] $p>3$, $n=p$, $0<r<p-1$ and $2\mid r$;
\item [\em(ii)] $p=2$, $2\nmid n$, $n\geq 2r>2$, $r\neq n-2,\ n-1$ and $n\equiv 2r+1 \pmod {2^L}$ where the integer $L$ satisfies $2^{L-1}\leq r<2^L$;
\item [\em (iii)] $p=2$, $2\nmid n$, $2r>n$, $r\neq n-2,\ n-1$ and $n\equiv 2r+1 \pmod {2^{L'}}$ where the integer $L'$ satisfies $2^{L'-1}\leq n-r-1<2^{L'}$.
\end{enumerate}
\end{thmx}

\begin{thmx}\label{T;C}
Let $p>2$ and $n$, $r$ be integers such that $n>0$ and $0\leq 2r\leq n$. Then $S^{(n-r,r)}$ is a trivial source module if and only if $(n-r,r)\in JM(n)_p$.
\end{thmx}
\begin{thmx}\label{T;B}
Let $p=2$ and $n$ be an integer such that $n\geq 4$. Let $\lambda$ be a partition of $n$ with $2$-weight $2$ and $\kappa_2(\lambda)$ be its $2$-core. Then $S^{\lambda}$ is a trivial source module if and only if $\lambda$ falls into one of the following cases:
\begin{enumerate}
\item [\em(i)] $\lambda\in JM(n)_2$;
\item [\em(ii)]$\lambda\notin JM(n)_2$, $S^\lambda\cong Y^\mu$ and $\mu=\kappa_2(\lambda)+(2,2)$.
\end{enumerate}
\end{thmx}

Notice that the case (ii) of Theorem \ref{T;B} indeed happens. For example, let $p=2$. It is well-known that $S^{(3,1,1)}$ is an indecomposable, non-simple, trivial source module. Theorem \ref{T;B} says that $S^{(3,1,1)}\cong Y^{(3,2)}$.

The paper is presented as follows. In Section $2$, we set up the notation used in the paper and give a brief summary of the required knowledge. Section $3$ is used to introduce fixed-point functors of symmetric groups and give some properties of them. Using these properties, in Sections $4$ and $5$, we prove Theorems \ref{T;A} and \ref{T;C} respectively. In the final section, we use Brou\'{e} correspondence to study some trivial source $\F \sym{n}$-modules and show Theorem \ref{T;B}.
\section{Notation and Preliminaries}
Throughout the whole paper, let $\F$ be an algebraically closed field with positive characteristic $p$. For a given finite group $G$, write $H\leq G$ $(\text{resp}.\ H<G)$ to indicate that $H$ is a subgroup (resp. a proper subgroup) of $G$. All $\F G$-modules considered in the paper are finitely generated left $\F G$-modules. Write $\uparrow$, $\downarrow$ to denote induction and restriction of modules respectively. Use $\otimes$ and $\boxtimes$ to present inner tensor product and outer tensor product of modules respectively. By abusing notation, also use $\F_G$ to denote the trivial $\F G$-module and omit the subscript if there is no confusion.
\subsection{Vertices, sources, Green correspondences and Scott modules}
The reader is assumed to be familiar with modular representation theory of finite groups. For a general background of the topic, one may refer to \cite{JAlperin1} or \cite{HNYT}.

Let $G$ be a finite group and $M$, $N$ be $\F G$-modules. Write $N\mid M$ if $N$ is isomorphic to a direct summand of $M$, i.e., $M\cong L\oplus N$ for some $\F G$-module $L$. If $N$ is indecomposable, for a decomposition of $M$ into a direct sum of indecomposable modules, the number of indecomposable direct summands that are isomorphic to $N$ is well-defined by the Krull-Schmidt Theorem and is denoted by $[M:N]$. Let $m$ be a positive integer and $N_i$ be an $\F G$-module for all $1\leq i\leq m$. Write $M\sim \sum_{i=1}^m N_i$ to indicate that there exists a filtration of $M$ such that all quotient factors are exactly isomorphic to these $N_i$'s.  The dual of $M$ is denoted by $M^*$.

Assume further that $M$ is an indecomposable $\F G$-module and $P\leq G$. Following \cite{JGreen}, say $P$ a vertex of $M$ if $P$ is a minimal (with respect to inclusion of subgroups) subgroup of $G$ subject to the condition that $M\mid ({M\downarrow_{P}}){\uparrow^G}$. All the vertices of $M$ are known to form a $G$-conjugacy class of $p$-subgroups of $G$. They are $G$-conjugate to a subgroup of a defect group of the $p$-block containing $M$. It is clear that $M$ and the indecomposable $\F G$-modules isomorphic to $M$ have same vertices. Moreover, note that $M$ and $M^*$ have same vertices by the definition. Let $P$ be a vertex of $M$. There exists some indecomposable $\F P$-module $S$ such that $M\mid S{\uparrow^G}$. It is called a $P$-source of $M$. All the $P$-sources of $M$ are $N_{G}(P)$-conjugate to each other, where $N_{G}(P)$ is the normalizer of $P$ in $G$. Call $M$ a trivial source module if all the $P$-sources of $M$ are trivial modules. Note that $M$ is a trivial source module if and only if $M^*$ is a trivial source module.

Let $N_G(P)\leq H\leq G$. The Green correspondent of $M$ with respect to $H$, denoted by $\mathcal{G}_{H}(M)$, is an indecomposable $\F H$-module with a vertex $P$ such that $\mathcal{G}_H(M)\mid M{\downarrow_H}$. It is well-known that $M{\downarrow_H}\cong \mathcal{G}_H(M)\oplus U$ and $\mathcal{G}_H(M){\uparrow^G}\cong M\oplus V$ where $U$ and $V$ are $\F H$-module and $\F G$-module respectively. Furthermore, for any indecomposable direct summand $W$ of $U$ or $V$, recall that $P$ is not a vertex of $W$. For more properties of $\mathcal{G}_H(M)$, one can refer to \cite[\S 4.4]{HNYT}.

Let $K\leq G$. For the transitive permutation module $\F_K{\uparrow^G}$, a well-known fact says that it has a unique trivial submodule and a unique trivial quotient module. Moreover, there exists an indecomposable direct summand of $\F_K{\uparrow^G}$ with a trivial submodule and a trivial quotient module. The module, unique up to isomorphism, is called the Scott module of $\F_K{\uparrow^G}$ and is denoted by $Sc_G(K)$. The vertices of $Sc_G(K)$ are $G$-conjugate to the Sylow $p$-subgroups of $K$ and it is a trivial source module.
For more properties of $Sc_G(K)$, one may refer to \cite{Burry}.
\subsection{Generic Jordan types and Brauer constructions of modules}
We shall introduce some tools used in the paper. Let $C_p=\langle g\rangle$ be a cyclic group of order $p$ and $M$ be an $\F C_p$-module. According to representation theory of cyclic groups, the matrix representations of $g$ on all indecomposable $\F C_p$-modules are exactly the Jordan blocks with all eigenvalues $1$ and sizes between $1$ to $p$. So $M$ can be viewed as a direct sum of the Jordan blocks. We call the direct sum of the Jordan blocks the Jordan type of $M$ and denote it by $[1]^{n_1}\ldots [p]^{n_p}$, where $[i]$ denotes a Jordan block of size $i$ and $n_i$ means the number of Jordan blocks of size $i$ in the sum.

Let $E$ be an elementary abelian $p$-group of order $p^n$ with a generator set $\{g_1,\ldots,g_n\}$ and $M$ be an $\F E$-module. Let $\K$ denote the extension field $\F(\alpha_1,\ldots,\alpha_n)$ over $\F$ where $\{\alpha_1,\ldots,\alpha_n\}$ is a set of indeterminates. Write $u_{\alpha}$ to denote the element $1+\sum_{i=1}^n\alpha_i(g_i-1)$ of $\K E$. Note that $\langle u_\alpha\rangle$ is a cyclic group of order $p$ and $M$ can be viewed as a $\K \langle u_\alpha\rangle$-module naturally. So the Jordan type of $M{\downarrow_{\langle u_\alpha\rangle}}$ is meaningful and is called generic Jordan type of $M$. Wheeler in \cite{WW} showed that generic Jordan type of $M$ is independent of the choices of generators of $E$. The stable generic Jordan type of $M$ is the generic Jordan type of $M$ modulo free direct summands.
Notice that the modules isomorphic to $M$ as $\F E$-modules and $M$ have same generic Jordan type. For more properties of generic Jordan types of modules, one may refer to \cite{EFJPAS} and \cite{WW}.

Let $G$ be a finite group and $M$ be an $\F G$-module. Say $M$ a $p$-permutation module if, for every $p$-subgroup $P$ of $G$, there exists a basis $\B_P$ of $M$ depending on $P$ such that $\B_P$ is permuted by $P$. Note that the class of $p$-permutation modules is closed by taking direct sums and non-zero direct summands. It is well-known that indecomposable $p$-permutation $\F G$-modules are exactly trivial source $\F G$-modules. Therefore, a $p$-permutation $\F G$-module is also equivalently defined to be a non-zero direct summand of a permutation $\F G$-module. To detect $p$-permutation $\F G$-modules, the following direct corollary of \cite[Lemma 3.1]{JLW} is helpful.
\begin{lem}\label{T;permutaiton}
Let $E$ be an elementary abelian $p$-subgroup of a finite group $G$. Let $M$ be a $p$-permutation $\F G$-module. Then $M{\downarrow_E}$ has stable generic Jordan type $[1]^m$ for some non-negative integer $m$.
\end{lem}

We now describe the Brauer constructions of $p$-permutation modules developed by Brou\'{e} in \cite{Broue}. Let $P\leq G$ and $M^P$ be the fixed-point subspace of an $\F G$-module $M$ under the action of $P$. Note that $M^P$ is an $\F [N_G(P)/P]$-module. For any $p$-subgroup $Q$ of $P$, the relative trace map from $M^Q$ to $M^P$, denoted by $\Tr_Q^P$, is defined to be
$$ \Tr_Q^P(m)=\sum_{g\in \{P/Q\}}gm,$$
where $\{P/Q\}$ is a complete set of representatives of left cosets of $Q$ in $P$. Observe that the linear map $\Tr_Q^P$ is well-defined as it is independent of the choices of $\{P/Q\}$. Continuously, $$\Tr^P(M)=\sum_{Q<P}\Tr_Q^P(M^Q).$$
It is also obvious to see that $\Tr^P(M)$ is an $\F[N_G(P)/P]$-submodule of $M^P$. The Brauer construction of $M$ with respect to $P$, written as $M(P)$, is defined to be the $\F[N_G(P)/P]$-module $$M^P/\Tr^P(M).$$ Notice that $M(P)=0$ if $P$ is not a $p$-subgroup of $G$. For our purpose, we collect some well-known properties of Brauer constructions of modules as follows.
\begin{lem}\label{L;Brauer}
The following assertions hold.
\begin{enumerate}
\item [\em(i)] Let $G$ be a finite group and $P$ be a $p$-subgroup of $G$. Let $M$ be a $p$-permutation $\F G$-module. If there exist some $\F G$-modules $U$ and $V$ such that $M\cong U\oplus V$, then $M(P)\cong U(P)\oplus V(P)$.
\item [\em(ii)] \cite[Lemma 2.7]{GLDM} Let $G_1$, $G_2$ be finite groups and $P_1$, $P_2$ be $p$-groups such that $P_1\leq G_1$ and $P_2\leq G_2$. Let $M_i$ be a $p$-permutation $\F G_i$-module for any $1\leq i\leq 2$. Then $(M_1\boxtimes M_2)(P_1\times P_2)\cong M_1(P_1)\boxtimes M_2(P_2)$ as modules under the action of $N_{G_1\times G_2}(P_1\times P_2)/(P_1\times P_2)\cong (N_{G_1}(P_1)/P_1)\times (N_{G_2}(P_2)/P_2)$.
\end{enumerate}
\end{lem}

The main result of this subsection, known as Brou\'{e} correspondence, is summarized as follows.
\begin{thm}\cite[Theorems 3.2 and 3.4]{Broue}\label{T;Broue}
Let $P$ be a $p$-subgroup of a finite group $G$ and $M$ be a trivial source $\F G$-module with a vertex $P$. Let $L$ be a $p$-permutation $\F G$-module.
\begin{enumerate}
\item [\em(i)] The correspondence $M\mapsto M(P)$ is a bijection between the isomorphism classes of trivial source $\F G$-modules with a vertex $P$ and the isomorphism classes of indecomposable projective $\F[\N_G(P)/P]$-modules.
\item [\em(ii)] The inflation of $M(P)$ from $N_G(P)/P$ to $N_G(P)$ is isomorphic to
    $\mathcal{G}_{N_G(P)}(M)$.
\item [\em(iii)] Let $N$ be a trivial source $\F G$-module with a vertex $Q$. Then $N\mid L$ if and only if $N(Q)\mid L(Q)$. We also have $[L:N]=[L(Q):N(Q)]$.
\end{enumerate}
\end{thm}
\subsection{Combinatorics}
Let $\mathbb{N}$ be the set of natural numbers. Let $n\in \mathbb{N}$ and $\mathbf{n}$ be the set $\{1,\ldots,n\}$. Let $\sym{n}$ be the symmetric group acting on $\mathbf{n}$. Set $\sym{0}$ to be the trivial group. A partition of a positive integer $m$ is a non-increasing sequence of positive integers $(\lambda_1,\ldots,\lambda_{\ell})$ such that $\sum_{i=1}^\ell\lambda_i=m$. By abusing notation, as the empty set, the unique partition of $0$ is denoted by $\varnothing$. Let $\lambda=(\lambda_1,\ldots,\lambda_\ell)$ be a partition. As usual, define $|\lambda|$ to be $\sum_{i=1}^\ell\lambda_i$ and write $\lambda\vdash |\lambda|$ to indicate that $\lambda$ is a partition of $|\lambda|$. Assume further that $\lambda\vdash n$. Say $\lambda$ a two-part partition if $\ell\leq 2$. By the exponential expression of partitions, $\lambda$ is called a hook partition if $\lambda=(n-r,1^r)$ for some integer $r$ such that $0\leq r<n$. We will use exponential expression of partitions throughout the whole paper. The Young diagram of $\lambda$, denoted by $[\lambda]$, is the set
$$ \{(i,j)\in \mathbb{N}^2:\ 1\leq i\leq \ell,\ 1\leq j\leq \lambda_i\}.$$
Each element of $[\lambda]$ is called a node of $[\lambda]$. For any node $(i,j)$ of $[\lambda]$, the hook of the node $(i,j)$ is the set $\{(i,k)\in [\lambda]:\ k\geq j\}\cup\{(k,j)\in [\lambda]:\ k>i\}$. Denote the set by $H_{i,j}^\lambda$ and put $h_{i,j}^\lambda=|H_{i,j}^\lambda|$. We will not distinguish between $\lambda$ and $[\lambda]$ in the paper.

The $p$-core $\kappa_p(\lambda)$ of $\lambda$ is the partition whose Young diagram is constructed by removing all rim $p$-hooks from $\lambda$ successively. Due to the Nakayama Conjecture, the $p$-cores of partitions of $n$ label the $p$-blocks of $\F \sym{n}$. We thus write $B_{\kappa_p(\lambda)}$ to denote the $p$-block of $\F\sym{n}$ labelled by $\kappa_p(\lambda)$. The number of rim $p$-hooks removed from $\lambda$ to obtain $\kappa_p(\lambda)$ is called the $p$-weight of $\lambda$ and is denoted by $w_p(\lambda)$. The $p$-weight of $B_{\kappa_p(\lambda)}$ is defined to be $w_p(\lambda)$. One can choose a defect group of $B_{\kappa_p(\lambda)}$ to be a Sylow $p$-subgroup of $\sym{pw_p(\lambda)}$. We say that $\lambda$ is $p$-restricted if $\lambda_\ell<p$ and $\lambda_i-\lambda_{i+1}<p$ for all $1\leq i<\ell$. Take $\lambda'$ to be the conjugate of $\lambda$ and say $\lambda$ a $p$-regular partition if $\lambda'$ is $p$-restricted. The $p$-adic expansion of $\lambda$ is the unique sum $\sum_{i=0}^mp^i\lambda(i)$ for some $m\in\mathbb{N}\cup\{0\}$, where $\lambda(i)$ is $p$-restricted or $\varnothing$ for all $0\leq i\leq m$, $\lambda(m)\neq \varnothing$ and $\lambda=\sum_{i=0}^mp^i\lambda(i)$ via component-wise addition and scalar multiplication of partitions.
\subsection{Modules of symmetric groups}
We now briefly present some material of representation theory of symmetric groups needed in the paper. One can refer to \cite{GJ1} or \cite{GJ3} for a background of the topic. Given the non-negative integers $m_1,\ldots, m_\ell$ such that $\sum_{i=1}^{\ell}m_i\leq n$, define $\sym{(m_1,\ldots, m_\ell)}=\sym{m_1}\times\cdots\times\sym{m_\ell}\leq\sym{n}$, where $\sym{m_1}$ acts on the set $\{1,\ldots, m_1\}$ (if $m_1\neq 0$), $\sym{m_2}$ acts on the set $\{m_1+1,\ldots, m_1+m_2\}$ (if $m_2\neq 0$) and so on. Let $\lambda\vdash n$. The Young subgroup associated with $\lambda$ of $\sym{n}$ is defined to be $\sym{\lambda}$. Throughout the paper, for a given integer $m$ with $1\leq m\leq n$, view $\sym{m}$ and $\sym{n-m}$ as corresponding components of $\sym{(m,n-m)}$. So they commute with each other. Use $sgn(n)$ to denote the sign module of $\sym{n}$ and omit the parameter if the symmetric group is clear. If $p=2$, $sgn$ means the corresponding trivial module.

To describe modules of symmetric groups, we assume that the reader is familiar with the definitions of tableau, tabloid and polytabloid (see \cite{GJ1} or \cite{GJ3}). The Young permutation module with respect to $\lambda$, denoted by $M^{\lambda}$, is the $\F$-space generated by all $\lambda$-tabloids where $\sym{n}$ permutes these tabloids. Notice that $M^\lambda$ is a $p$-permutation module since it is a permutation module.

The Specht module $S^\lambda$ is an $\F\sym{n}$-submodule of $M^{\lambda}$ generated by all $\lambda$-polytabloids. The dual of $S^\lambda$ is denoted by $S_\lambda$. It is well-known that $S^{\lambda}$ is indecomposable if $p>2$ or $p=2$ and $\lambda$ is $2$-regular (or $2$-restricted). The relation of $S^{\lambda}$ and $S_{\lambda'}$ is given by $S^\lambda\otimes sgn\cong S_{\lambda'}$. In particular, if $p=2$, $S^\lambda\cong S^{\lambda'}$ if $S^\lambda$ is self-dual. Unlike the characteristic zero case, $\{S^\lambda:\ \lambda\vdash n\}$ is usually not the set of all simple modules of $\sym{n}$ up to isomorphism. However, when $\lambda$ is $p$-regular, James in \cite[Theorem 11.5]{GJ1} showed that $S^\lambda$ has a simple head $D^\lambda$. Moreover, the set $D_{n,p}=\{D^{\mu}:\ \mu\ \text{is}\ \text{a}\ \text{$p$-regular}\ \text{partition}\ \text{of}\ n\}$
forms a complete set of representatives of isomorphic classes of simple $\F\sym{n}$-modules. It is well-known that every module in $D_{n,p}$ is self-dual. Following the notation used in \cite[Section 4]{DanzLim}, we denote the set $\{\lambda\vdash n:\ S^\lambda\ \text{is}\ \text{simple}\}$ by $JM(n)_p$.
A Specht filtration of an $\F\sym{n}$-module is a filtration of the module with all the successive quotient factors isomorphic to some Specht modules.

For a given decomposition of $M^\lambda$ into a direct sum of indecomposable modules, there exists an indecomposable direct summand of $M^\lambda$ that contains $S^\lambda$. The module is unique up to isomorphism and is denoted by $Y^\lambda$. Let $\mu\vdash n$ and $\unlhd$ denote the dominance order on the set of all partitions of $n$. James in \cite[Theorem 3.1]{GJ2} showed that $[M^\lambda:Y^\lambda]=1$ and $[M^{\lambda}:Y^{\mu}]\neq 0$ only if $\lambda\unlhd \mu$. One thus gets
\begin{equation*}
M^\lambda\cong Y^\lambda\oplus\bigoplus_{\mu\rhd\lambda}k_{\lambda,\mu}Y^\mu,
\end{equation*}
where $k_{\lambda,\mu}=[M^\lambda:Y^\mu]$ and all $k_{\lambda,\mu}$'s are known as $p$-Kostka numbers. Note that Young modules are self-dual and are $p$-permutation modules with trivial sources.
The Brauer constructions of Young modules with respect to their vertices are studied by \cite{Erdmann} and \cite{GJ}. To describe them, let $\lambda$ have $p$-adic expansion $\sum_{i=0}^mp^i\lambda(i)$ for some non-negative integer $m$, where $|\lambda(i)|=n_i$ for all $0\leq i\leq m$. Denote by $\O_\lambda$ the partition
$$(\underbrace{p^m,\ldots,p^m}_{n_m\ \text{times}},\ \underbrace{p^{m-1},\ldots,p^{m-1}}_{n_{m-1}\ \text{times}},
\ldots,\ \underbrace{1,\ldots,1}_{n_0\ \text{times}}).$$
\begin{thm}\label{T; Youngvertices}\cite{Erdmann, GJ}
Let $\lambda\vdash n$. Then the vertices of $Y^{\lambda}$ are $\sym{n}$-conjugate to the Sylow $p$-subgroups of $\mathfrak{S}_{\O_\lambda}$. Moreover, $Y^\lambda$ is projective if and only if $\lambda$ is $p$-restricted.
\end{thm}
\begin{thm}\cite{Erdmann}\label{Erdmann}
Let $\lambda\vdash n$ and $\lambda$ have $p$-adic expansion $\sum_{i=0}^mp^i\lambda(i)$ where $|\lambda(i)|=n_i$ for all $0\leq i\leq m$. Let $\alpha=(n_0,\ldots,n_m)$ and $P_{\lambda}$ be a Sylow $p$-subgroup of $\sym{\O_{\lambda}}$.
\begin{enumerate}
\item [\em(i)] We have $\sym{\alpha}\cong \N_{\sym{n}}(P_\lambda)/\N_{\sym{\O_\lambda}}(P_\lambda)\cong (\N_{\sym{n}}(P_\lambda)/P_\lambda)/(\N_{\sym{\O_\lambda}}(P_\lambda)/P_\lambda).$
\item [\em(ii)] We have that $\N_{\sym{\O_\lambda}}(P_\lambda)/P_\lambda$ acts trivially on $Y^\lambda(P_\lambda)$. Therefore, $Y^\lambda(P_\lambda)$ can be viewed as an $\F\sym{\alpha}$-module.
\item [\em(iii)] We have $Y^{\lambda}(P_\lambda)\cong Y^{\lambda(0)}\boxtimes
Y^{\lambda(1)}\boxtimes\cdots\boxtimes Y^{\lambda(m)}$ as $\F \sym{\alpha}$-modules.
\end{enumerate}
\end{thm}
We conclude the section by collecting some results used in the following discussion.
\begin{thm}\label{T;Hemmer}
\cite[Proposition 4.6]{DKZ} The trivial source simple $\F\sym{n}$-modules are exactly the simple $\F\sym{n}$-Specht modules.
\end{thm}

\begin{lem}\label{L;p-permutation modules}
Let $\lambda\vdash n$ and $M$ be an $\F\sym{n}$-module. Then
\begin{enumerate}
\item [\em(i)]$M$ is a $p$-permutation module if and only if $M^*$ is a $p$-permutation module.
\item [\em(ii)] $M$ is a $p$-permutation module if and only if $M\otimes sgn$ is a $p$-permutation module, where $sgn$ means the $\F\sym{n}$-trivial module if $p=2$.
\item [\em(iii)] $S^\lambda$ is a $p$-permutation module if and only if $S^{{\lambda}'}$ is a $p$-permutation module.
\end{enumerate}
\end{lem}
\begin{proof}
Note that $*$ preserves direct sum of $\F\sym{n}$-modules and takes a permutation module to a permutation module. (i) thus follows by noticing that $M\cong (M^{*})^*$.
Let $v$ be a generator of $sgn$. If $M$ is a $p$-permutation module, for any $p$-subgroup $P$ of $\sym{n}$, let $\B_P$ be a basis of $M$ which is permuted by $P$. Note that $\B_{P,v}=\{u\otimes v:\ u\in \B_P\}$ is a basis of $M\otimes sgn$ which is permuted by $P$. So $M\otimes sgn$ is a $p$-permutation module by the definition. (ii) follows by noting that $M\cong(M\otimes sgn)\otimes sgn$.
As $S^{\lambda'}\cong (S^\lambda\otimes sgn)^*$, (iii) thus follows by combining (i) and (ii). The lemma follows.
\end{proof}
\section{Fixed-point functors of symmetric groups}
The section is designed to study fixed-point functors of symmetric groups. After a short summary of these objects, we present some properties of the functors. Some results will be used in the following two sections while others may be of independent interest.

Let $m\in \mathbb{N}$, $m\leq n$ and $M$ be an $\F\sym{n}$-module. Recall that $\sym{n-m}$ centralizes $\sym{m}$. So the space $M^{\sym{m}}$ is viewed as an $\F\sym{n-m}$-module naturally. The fixed-point functor with the parameter $m$, introduced by Hemmer in \cite{Hemmer2}, is defined to be
\begin{align*}
\mathcal{F}_m: \F\sym{n}\text{-mod}\mapsto\F\sym{n-m}\text{-mod}
\end{align*}
setting $\mathcal{F}_m(M)\!=\!M^{\sym{m}}\cong
\text{Hom}_{\F\sym{m}}(\F,M{\downarrow_{\sym{m}}})\!\cong\!\text{Hom}_{\F\sym{n}}(M^{(m,1^{n-m})},M).$ It is also convenient to regard $\mathcal{F}_m(M)$ as the largest subspace of $M{\downarrow_{\sym{(m,n-m)}}}$ where $\sym{m}$ acts trivially. Note that $\mathcal{F}_1(M)=M{\downarrow_{\sym{n-1}}}$ by the definition. Moreover, the functor $\mathcal{F}_m$ is exact if $m<p$ and is not exact in general if $m\geq p$. In \cite{Hemmer1}, Hemmer obtained some results of extensions of simple modules of symmetric groups by using fixed-point functors. For any $\lambda\vdash n$, he also conjectured that
$\mathcal{F}_m(S^\lambda)$ had a Specht filtration (see \cite[Conjecture 7.3]{Hemmer2} or \cite[Conjecture 7.2]{Hemmer1}). However, the conjecture was shown to be false in \cite{DG}.

We now begin to prove some properties of fixed-point functors. Throughout the whole section, fix an integer $m$ satisfying $1\leq m\leq n$.
\begin{lem}\label{L;Directsum}
Let $M$ and $N$ be $\F\sym{n}$-modules. Then $\mathcal{F}_m(M\oplus N)=\mathcal{F}_m(M)\oplus \mathcal{F}_m(N)$.
\end{lem}
\begin{proof}
Note that $(M\oplus N)^{\sym{m}}=M^{\sym{m}}\oplus N^{\sym{m}}$. The lemma follows by restricting both sides of the equality to $\sym{n-m}$ and the definition of fixed-point functors.
\end{proof}
\begin{prop}\label{P;Projective}
Let $P$ be a projective $\F\sym{n}$-module. Then $\mathcal{F}_m(P)$ is a projective $\F\sym{n-m}$-module.
\end{prop}
\begin{proof}
We may assume that $P$ is non-zero. As $P$ is projective, for some positive integer $\ell$, we have
\begin{align}
(P{\downarrow_{\sym{(m,n-m)}}})^{\sym{m}}\cong\bigoplus_{i=1}^\ell(P_i\boxtimes Q_i)^{\sym{m}}=\bigoplus_{i=1}^\ell(P_i^{\sym{m}}\boxtimes Q_i),
\end{align}
where $P_i$ and $Q_i$ are indecomposable projective $\F\sym{m}$-module and $\F\sym{n-m}$-module respectively for all $1\leq i\leq \ell$. Regarding both sides of $(3.1)$ as $\F\sym{n-m}$-modules, we get that $\mathcal{F}_m(P)\cong\bigoplus_{i=1}^\ell \dim_\F {P_i}^{\sym{m}}. Q_i$, which implies the desired result.
\end{proof}
Let $\ell$ be an integer and $\Omega^\ell$ be the $\ell$th Heller translate of modules of a given symmetric group. We have
\begin{cor}\label{C;Heler}
Let $M$ be an $\F\sym{n}$-module. If $m<p$, for any integer $\ell$, we have $\mathcal{F}_m(\Omega^\ell(M))\cong \Omega^\ell(\mathcal{F}_m(M))$ in the stable category of $\F\sym{n-m}$-modules.
\end{cor}
\begin{proof}
The case $\ell=0$ is clear by Lemma \ref{L;Directsum} and Proposition \ref{P;Projective}. When $\ell\neq 0$, note that $\mathcal{F}_m$ is exact. By Proposition \ref{P;Projective} again, $\mathcal{F}_m$ takes a minimal projective resolution (resp. a minimal injective resolution) of $M$ to be a projective resolution (resp. an injective resolution) of $\mathcal{F}_m(M)$. Therefore, the desired result follows.
\end{proof}
Corollary \ref{C;Heler} may not be true if the fixed-point functor is not exact. For example, let $p=2$ and $n=4$. Note that $\mathcal{F}_2$ is not exact. Let $P$ be the projective cover of the trivial $\F\sym{4}$-module $\F$. For the short exact sequence
$0\rightarrow\Omega^1(\F)\rightarrow P\rightarrow\F\rightarrow 0$, by Proposition \ref{P;Projective} and long exact sequence,
we have $$0\rightarrow\mathcal{F}_2(\Omega^1(\F))\rightarrow\mathcal{F}_2(P)\rightarrow\F\rightarrow \text{Ext}_{\F\sym{4}}^1(M^{(2,1^2)}, \Omega^1(\F))\rightarrow 0.$$
As $\Omega^1(\F){\downarrow_{\sym{2}}}$ is not free, $\text{Ext}_{\F\sym{4}}^1(M^{(2,1^2)}, \Omega^1(\F))\cong \text{Ext}_{\F\sym{2}}^1(\F, \Omega^1(\F){\downarrow_{\sym{2}}})\neq 0$. These
facts and the long exact sequence imply that $\mathcal{F}_2(\Omega^1(\F))\cong \mathcal{F}_2(P)$. So $\mathcal{F}_2(\Omega^1(\F))$ is projective by Proposition \ref{P;Projective}. However, the $\F\sym{2}$-module $\Omega^1(\F)$ is not projective. We thus have $\mathcal{F}_2(\Omega^1(\F))\not\cong \Omega^1(\F)$ in the stable category of $\F\sym{2}$-modules.
\begin{lem}\label{L;permutation1}
Let $H\leq \sym{n}$. Then $\mathcal{F}_m(\F_H{\uparrow^{\sym{n}}})$ is a permutation $\F\sym{n-m}$-module.
\end{lem}
\begin{proof}
It is sufficient to check that $\mathcal{F}_m(\F_H{\uparrow^{\sym{n}}})$ has a permutation basis under the action of $\sym{n-m}$. Let $s=|\sym{n}:H|$ and $\B=\bigcup_{i=1}^s\{g_iH\}$ be a complete set of representatives of left cosets of $H$ in $\sym{n}$.
For any $\sigma\in \sym{n}$ and $g_iH\in \B$, there exists a unique $g_jH\in \B$ such that $\sigma g_iH=g_jH$.
We denote $g_j$ by $g_{\sigma,i}$ and view $\B$ as a basis of $\F_H{\uparrow^{\sym{n}}}$. For all $1\leq i\leq s$, let $\O(g_i)$ be the orbit of $\B$ containing $g_iH$
under the action of $\sym{m}$ and $\O(g_i)=\{x_{i,1}g_iH,\ldots, x_{i,t_i}g_iH\}$. So  $\B=\biguplus_{j=1}^{t}\O(g_{i_j})$,
where, for all $1\leq j\leq t\leq s$, $t_{i_j}=|\O(g_{i_j})|$ and $x_{i_j,1},\ldots, x_{i_j,t_{i_j}}\in\sym{m}$.
For any $\sigma\in \sym{n}$, let $\sigma\O(g_i)=\{\sigma xH:\ xH\in\O(g_i)\}$ for all $1\leq i\leq s$.
Moreover, write $\overline{\O(g_i)}$ to be the sum $\sum_{xH\in \O(g_i)}xH$ and notice that $\overline{\B}=\bigcup_{j=1}^t\{\overline{\O(g_{i_j})}\}$
is a basis of $(\F_H{\uparrow^{\sym{n}}})^{\sym{m}}$. For any $\tau\in\sym{n-m}$, $\O(g_{i_j})\subseteq \B$, if $|\O(g_{i_j})|>1$, for any integers $u,v$ such that $1\leq u<v\leq t_{i_j}$, note that $x_{i_j,u}g_{\tau,i_j}H\neq x_{i_j,v}g_{\tau,i_j}H$ as $x_{i_j,u}g_{i_j}H\neq x_{i_j,v}g_{i_j}H$. Therefore, $\tau\O(g_{i_j})\subseteq\O(g_{\tau,i_j})$. Similarly, we have $\tau^{-1}\O(g_{\tau,i_j})\subseteq\O(g_{i_j})$. The two relations imply that $\tau\O(g_{i_j})=\O(g_{\tau,i_j})=\O(g_{i_a})$ for some $a\in \mathbb{N}$ and $a\leq t$. We thus have
$$\tau\overline{\O(g_{i_j})}=\tau\sum_{k=1}^{t_{i_j}}x_{i_j,k}g_{i_j}H=\sum_{k=1}^{t_{i_j}}x_{i_j,k}\tau g_{i_j}H=\sum_{k=1}^{t_{i_j}}x_{i_j,k}g_{\tau,i_j}H=\overline{\O(g_{\tau,i_j})}=\overline{\O(g_{i_a})}.$$
If $|\O(g_{i_j})|=1$, by a similar deduction, $\tau\overline{\O(g_{i_j})}=\overline{\O(g_{i_b})}$ where $b\in \mathbb{N}$, $b\leq t$ and $|\O(g_{i_b})|=1$. So $\overline{\B}$ is a permutation basis under the action of $\sym{n-m}$. \end{proof}

We close the section by a corollary. It will be used in the following discussion.
\begin{cor}\label{L;permutation}
Let $M$ be a $p$-permutation $\F\sym{n}$-module. If $\mathcal{F}_m(M)\neq 0$, then $\mathcal{F}_m(M)$ is a $p$-permutation $\F\sym{n-m}$-module.
\end{cor}
\begin{proof}
View $M$ as a direct summand of a direct sum of some transitive permutation $\F\sym{n}$-modules. By Lemmas \ref{L;Directsum} and \ref{L;permutation1}, observe that $\mathcal{F}_m(M)$ is a non-zero direct summand of a direct sum of some permutation $\F\sym{n-m}$-modules. So $\mathcal{F}_m(M)$ is indeed a $p$-permutation $\F\sym{n-m}$-module by the definition.
\end{proof}
\section{The Specht modules labelled by hook partitions}
The section focuses on proving Theorem \ref{T;A}. Throughout the whole section, fix $n>1$ and integers $m,r$ satisfying $1\leq m, r<n$. We first present some explicit calculation results of $\mathcal{F}_m(S^{(n-r,1^r)})$ and then finish the proof of Theorem \ref{T;A}.
\begin{nota}\label{N;Notation1}
For convenience, we introduce the required notation.
\begin{enumerate}[(i)]
\item Let $S$ be a set and write $\langle S\rangle$ to denote an $\F$-linear space spanned by elements of $S$. Set $\langle\varnothing\rangle=0$ and $\langle s\rangle=\langle S\rangle$ if $S=\{s\}$. Let $A_n$ be a $n$-dimensional linear space over $\F$ with a basis $\{e_1,\ldots,e_n\}$ and  $A_b=\langle\{e_1,\ldots,e_b\}\rangle$ for all $1\leq b\leq n$. Write $f_i=e_i-e_n$ for all $1\leq i\leq n$. Note that $A_b$ is an $\F\sym{b}$-module as $\sym{b}$ permutes the subscripts of its basis $\{e_1,\ldots,e_b\}$. So $A_n\cong M^{(n-1,1)}$ and $S^{(n-1,1)}\cong\langle\{f_1,\ldots,
    f_{n-1}\}\rangle$. Identify a basis of $S^{(n-1,1)}$ with $\{f_1,\ldots,f_{n-1}\}$. If $n-m\geq 2$, also identify a basis of $S^{(n-m-1,1)}$ with $\{f_{m+1},\ldots,f_{n-1}\}$.
\item Let $a,b,c\in \mathbb{N}$. Set $J(a,b,c)=\{\underline{\mathbf{i}}=(i_1,\ldots,i_b)\in\mathbb{N}^b:\  a\leq i_1<\cdots<i_b\leq c\}$ and $J^+(a,b,c)=\{\underline{\mathbf{i}}=(i_1,\ldots,i_b)\in\ \mathbb{N}^b:\ a\leq i_1<\cdots<i_b<c\}$. Write $J(a,b)$, $J^+(a,b)$ to denote $J(a,b,n)$, $J^+(a,b,n)$ respectively. Let $b\leq n$ and $\bigwedge^cA_b$, $\bigwedge^c S^{(n-1,1)}$ be the $c$th exterior power of $A_b$, $S^{(n-1,1)}$ respectively. By convention, put $\bigwedge^0S^{(n-1,1)}=\F$ and $\bigwedge^1S^{(n-1,1)}=S^{(n-1,1)}$. If $c<n$ and $c\leq b$, then $\bigwedge^c A_b$ has a basis
    $\{e_{\underline{\mathbf{i}}}=e_{i_1}\wedge\cdots\wedge e_{i_c}:\ \underline{\mathbf{i}}=(i_1,\ldots,i_c)\in J(1,c,b)\}$
    and $\bigwedge^cS^{(n-1,1)}$ has a basis $\{f_{\underline{\mathbf{i}}}=f_{i_1}\wedge\cdots\wedge f_{i_c}:\ \underline{\mathbf{i}}=(i_1,\ldots,i_c)\in J^+(1,c)\}$. Use $\B_{A_b}^c$, $\B_S^c$ to denote the two bases respectively. If $r+m-n<c<r$ and $n-m\geq 2$, then the $(r-c)$th exterior power $\bigwedge^{r-c}S^{(n-m-1,1)}$ has a basis  $\{f_{\underline{\mathbf{i}}}:\ \underline{\mathbf{i}}\in J^+(m+1,r-c)\}$. For any $v\in\bigwedge^rS^{(n-1,1)}$ and $f_{\underline{\mathbf{i}}}\in \B_S^r$, let $ d_{\underline{\mathbf{i}}}^{v}$ be the coefficient of $f_{\underline{\mathbf{i}}}$ in the linear combination of elements of $\B_S^r$ representing $v$.
\item  Let $\ell$ be an integer such that $0\leq \ell\leq r$. Recall that $\mathbf{m}=\{1,\ldots,m\}$ and set
 $$B_\ell=\{f_{\underline{\mathbf{i}}}=f_{i_1}\wedge\cdots\wedge f_{i_r}\in\B_S^r:\ |\{i_1,\ldots,i_r\}\cap\mathbf{m}|=\ell\}.$$
Note that $B_{u}\cap B_{v}=\varnothing$ if $u\neq v$ and observe that $B_{\ell}\neq\varnothing$ if and only if
\begin{align}
r+m-n+1\leq \ell\leq m.
\end{align}
For any $v\in (\bigwedge^rS^{(n-1,1)})^{\sym{m}}$ and $B_{\ell}\neq\varnothing$, put $\overline{B_{\ell}^v}=\sum_{w\in B_{\ell}} d_{w}^{v}w,$ where $d_{w}^v=d_{\underline{\mathbf{i}}}^v$ if $w=f_{\underline{\mathbf{i}}}$. For completeness, let $\overline{B_{\ell}^v}=0$ if $B_{\ell}=\varnothing$. Then $v=\sum_{i=0}^r\overline{B_{i}^v}.$ Moreover, for any $g\in\sym{m}$,
$g\overline{B_{i}^v}=\overline{B_{i}^v}$ for all $0\leq i\leq r$ as $gv=v$.
\item Let $b,c\in \mathbb{N}$ such that $b,c<n$. Set $v_{b,c}=\sum_{\underline{\mathbf{i}}\in J(1,c,b)}f_{\underline{\mathbf{i}}}$ if $c\leq b$ and put $v_{b,c}=0$ if $b<c$. So $v_{b,c}\in\bigwedge^c S^{(n-1,1)}$. Note that $\langle \{gv_{b,c}:\ g\in\F\sym{b}\}\rangle$ is isomorphic to an $\F\sym{b}$-submodule of $\bigwedge^c A_b$. In particular, $\langle v_{b,c}\rangle\subseteq \bigwedge^c A_b$ up to $\F\sym{b}$-isomorphism.
\end{enumerate}
\end{nota}

We now present our calculation results.
\begin{lem}\label{L;S(n-1,1)}
We have
\[\mathcal{F}_m(S^{(n-1,1)})\cong\begin{cases} S^{(n-m-1,1)}\oplus\F, & \text{if}\ p\mid m,\ 1<n-m,\\
M^{(n-m-1,1)}, &\text{if}\ p\nmid m,\ 1<n-m,\\
\F, &\text{if}\ n-m=1.\end{cases}
\]
\end{lem}
\begin{proof}
When $n-m\geq 2$, note that
$$(S^{(n-1,1)})^{\sym{m}}=\langle\{f_{m+1},\ldots,f_{n-1}\}\rangle\oplus\langle
f_1+\cdots+f_m\rangle.$$
If $p\mid m$, we get that $\sum_{i=1}^mf_i=\sum_{i=1}^me_i$ and
\begin{align*}
\mathcal{F}_m(S^{(n-1,1)})&=\langle\{f_{m+1},\ldots,f_{n-1}\}\rangle\oplus\langle
e_1+\cdots+e_m\rangle\\
&\cong S^{(n-m-1,1)}\oplus\F.
\end{align*}
If $p\nmid m$, denote $\frac{1}{m}\sum_{i=1}^me_i$ by $sum$ and have
\begin{align*}
\mathcal{F}_m(S^{(n-1,1)})=\langle\{sum-e_{m+1},\ldots,
sum-e_n\}\rangle
\cong M^{(n-m-1,1)}.
\end{align*}
When $n-m=1$, we get $\mathcal{F}_m(S^{(n-1,1)})=\langle\sum_{i=1}^{n-1}f_i\rangle$,
which is a trivial module. The lemma now follows.
\end{proof}
For further calculation, recall that $S^{(n-s,1^s)}\cong\bigwedge^s S^{(n-1,1)}$ for all $0\leq s<n$ (see \cite[Proposition 2.3]{Muller}).
\begin{lem}\label{L;Afixedpoints}
Let $p>2$, $b,c\in \mathbb{N}$ and $c\leq b\leq n$. We have
\[(\bigwedge^cA_b)^{\sym{b}}=\begin{cases} 0, & \text{if}\ c>1,\\
\langle \{e_1+\cdots+e_b\}\rangle, & \text{if}\ c=1.
\end{cases}\]
\end{lem}
\begin{proof}
For any $u\in (\bigwedge^cA_b)^{\sym{b}}$ and $u=\sum_{\underline{\mathbf{i}}\in J(1,c,b)}k_{\underline{\mathbf{i}}} e_{\underline{\mathbf{i}}}$, we have $k_{\underline{\mathbf{i}}}=k_{\underline{\mathbf{j}}}$ for any $\underline{\mathbf{i}}$, $\underline{\mathbf{j}}\in J(1,c,b)$ by using suitable permutations to act on $u$ and comparing the coefficients. So $(\bigwedge^cA_b)^{\sym{b}}\subseteq\langle\sum_{\underline{\mathbf{i}}\in J(1,c,b)}e_{\underline{\mathbf{i}}}\rangle$. When $c>1$, set $v=\sum_{\underline{\mathbf{i}}\in J(1,c,b)}e_{\underline{\mathbf{i}}}$ and get that $v=w+\sum_{(1,2,\ldots,i_c)\in J(1,c,b)}e_1\wedge e_2 \wedge\cdots\wedge e_{i_c}$. Note that no basis element $e_1\wedge e_2\wedge\cdots\wedge e_{j_c}$ of $\B_{A_b}^c$ is involved in $w$. Now use the transposition $(1,2)$ to act on both sides of the equality of $v$ and obtain that $(1,2)v=(1,2)w-\sum_{(1,2,\ldots,i_c)\in J(1,c,b)}e_1\wedge e_2\wedge\cdots\wedge e_{i_c}.$
So $(1,2)v\neq v$ and $(\bigwedge^cA_b)^{\sym{b}}=0$. When $c=1$, it is clear that $A_b^{\sym{b}}=\langle e_1+\cdots+e_b\rangle$. The lemma follows.
\end{proof}
\begin{lem}\label{L;p>2,calculation}
Let $p>2$. If $r>1$, for any $v\in(\bigwedge^rS^{(n-1,1)})^{\sym{m}}$, we have
\[v=\begin{cases} \overline{B_{0}^v}+\overline{B_{1}^v}, & \text{if}\ r<n-m,\\
\overline{B_{1}^v}, &\text{if}\ r=n-m,\\
0, & \text{if}\ r>n-m,\end{cases}\]
where $\overline{B_{1}^v}=\sum_{\underline{\mathbf{i}}\in J^+(m+1,r-1)}k_{\underline{\mathbf{i}}}v_{m,1}\wedge f_{\underline{\mathbf{i}}}$
and $k_{\underline{\mathbf{i}}}\in\F$ for all $\underline{\mathbf{i}}\in J^+(m+1,r-1)$.
\end{lem}
\begin{proof}
Recall that $v=\sum_{i=0}^r\overline{B_{i}^v}$. Moreover, $\overline{B_{i}^v}\in (\bigwedge^rS^{(n-1,1)})^{\sym{m}}$ for all $0\leq i\leq r$ as $v\in(\bigwedge^rS^{(n-1,1)})^{\sym{m}}$. We claim that $\overline{B_{i}^v}=0$ if $1<i\leq r$. If $B_i\neq \varnothing$, by the definition and an easy calculation, we get that
\begin{equation}
\overline{B_{i}^v}\in\begin{cases} \langle v_{m,i}\wedge w_i\rangle, & \text{if}\ 0<i<r,\\
 \langle v_{m,r}\rangle, & \text{if}\ i=r,\end{cases}
\end{equation}
where
$w_i\in\langle\{f_{\underline{\mathbf{i}}}:\ \underline{\mathbf{i}}\in J^+(m+1,r-i)\}\rangle$.
Suppose that $\overline{B_{u}^v}\neq 0$ for some $1<u\leq r$. We have $u\leq m$, $v_{m,u}\neq 0$ and $w_u\neq 0$ if $u<r$. Notice that $\langle v_{m,u}\rangle\subseteq \bigwedge^u A_m$ up to $\F\sym{m}$-isomorphism. By Lemma \ref{L;Afixedpoints}, we get that $v_{m,u}$ is not fixed by $\sym{m}$ as $(\bigwedge^u A_m)^{\sym{m}}=0$. In particular, $v_{m,u}\wedge w_u\notin(\bigwedge^rS^{(n-1,1)})^{\sym{m}}$ if $1<u<r$ and $v_{m,r}\notin(\bigwedge^rS^{(n-1,1)})^{\sym{m}}$ if $u=r$. By $(4.2)$, these facts imply that $\overline{B_{u}^v}=0$, which is a contradiction. The claim is shown. When $r<n-m$, by $(4.1)$, note that $B_{0}\neq \varnothing$ and $B_{1}\neq \varnothing$. We therefore get $v=\overline{B_{0}^v}+\overline{B_{1}^v}$ by the claim. For the left cases, we can determine whether $B_{0}$ or $B_{1}$ is empty by the given conditions and $(4.1)$. Therefore, we obtain the corresponding results by the claim. The proof is now complete.
\end{proof}

\begin{lem}\label{L;F(S),p>2}
Let $p>2$. If $r>1$, then we have
\[\mathcal{F}_m(S^{(n-r,1^r)})\cong\begin{cases} M, & \text{if}\ p\mid m,\ r<n-m,\\
N, & \text{if}\ p\nmid m,\ r<n-m,\\
sgn, & \text{if}\ r=n-m,\\
0, & \text{if}\ r>n-m,
\end{cases}\]
where $M\cong S^{(n-m-r,1^r)}\oplus S^{(n-m-r+1,1^{r-1})}$ and $N\sim S^{(n-m-r,1^r)}+S^{(n-m-r+1,1^{r-1})}$.
\end{lem}
\begin{proof}
It is sufficient to determine $\mathcal{F}_m(\bigwedge^rS^{(n-1,1)})$. When $r<n-m$, define
\begin{align*}
\mathcal{C}_{0}=\{f_{\underline{\mathbf{i}}}:\ \underline{\mathbf{i}}\in J^+(m+1,r)\},\
\mathcal{C}_{1}=\{v_{m,1}\wedge f_{\underline{\mathbf{j}}}:\ \underline{\mathbf{j}}\in J^+(m+1,r-1)\}.
\end{align*}
Observe that the vectors of $\mathcal{C}_{0}\cup\mathcal{C}_{1}$ are all fixed by $\sym{m}$. Moreover, $\mathcal{C}_{0}\cup\mathcal{C}_{1}$ is linearly independent over $\F$. By Lemma \ref{L;p>2,calculation}, for any $v\in(\bigwedge^rS^{(n-1,1)})^{\sym{m}}$, $v=\overline{B_{0}^v}+\overline{B_{1}^v}$, $\overline{B_{0}^v}\in\langle \mathcal{C}_0\rangle$ and $\overline{B_{1}^v}=\sum_{\underline{\mathbf{j}}\in J^+(m+1,r-1)}k_{\underline{\mathbf{j}}}v_{m,1}\wedge f_{\underline{\mathbf{j}}},$ where $k_{\underline{\mathbf{j}}}\in \F$ for all $\underline{\mathbf{j}}\in J^+(m+1,r-1)$. We thus conclude that $\mathcal{C}_0\cup\mathcal{C}_1$ is a basis of $(\bigwedge^rS^{(n-1,1)})^{\sym{m}}$ by these facts. If $p\mid m$, $\mathcal{C}_{1}$ degenerates to be $\{(\sum_{i=1}^me_i)\wedge f_{\underline{\mathbf{j}}}:\ \underline{\mathbf{j}}\in J^+(m+1,r-1)\}$. Therefore, we get $\mathcal{F}_m(\bigwedge^rS^{(n-1,1)})=\langle \mathcal{C}_{0}\rangle\oplus\langle\mathcal{C}_{1}\rangle$, where $\langle \mathcal{C}_{0}\rangle\cong S^{(n-m-r,1^r)}$ and $\langle \mathcal{C}_{1}\rangle\cong S^{(n-m-r+1, 1^{r-1})}$ as $\F\sym{n-m}$-modules.
If $p\nmid m$, let $S=\mathcal{F}_m(\bigwedge^rS^{(n-1,1)})$ and $P=\langle\mathcal{C}_{0}\rangle$. Notice that $P\cong S^{(n-m-r,1^r)}$ as $\F\sym{n-m}$-modules. We need to show that $S/P\cong S^{(n-m-r+1, 1^{r-1})}$ as $\F\sym{n-m}$-modules. For any $v\in S$, let $\overline{v}$ be the image of $v$ under the natural map from $S$ to $S/P$. Define a linear map $\phi$ from $S/P$ to $\bigwedge^{r-1}S^{(n-m-1,1)}$ by sending each $\overline{v_{m,1} \wedge f_{\underline{\mathbf{j}}}}$ to $f_{\underline{\mathbf{j}}}$. It is obvious to see that $\phi$ is a linear isomorphism. To show that it is an $\F\sym{n-m}$-isomorphism, it is enough to check that the action of transposition $(n-1,n)$ is preserved by $\phi$. Set $\sigma=(n-1,n)$ and $s=\sum_{i=1}^me_i$. Note that $v_{m,1}=s-me_n$ by the definition. For any $v_{m,1}
\wedge f_{\underline{\mathbf{j}}}\in\mathcal{C}_1$ where $\underline{\mathbf{j}}=(j_1,\ldots,j_{r-1})$, we distinguish with two cases.
\begin{enumerate}[\text{Case} 1:]
\item $j_{r-1}<n-1.$
\end{enumerate}
We have
\begin{align*}
&\phi(\sigma(\overline{v_{m,1}\wedge f_{j_1}\wedge\cdots\wedge f_{j_{r-1}}}))\\
&=\phi(\overline{(s-me_{n-1})\wedge (e_{j_1}-e_{n-1})\wedge\cdots\wedge(e_{j_{r-1}}-e_{n-1})})\\
&=\phi(\overline{(s-me_n+me_n-me_{n-1})\wedge
(e_{j_1}-e_{n-1})\wedge\cdots\wedge(e_{j_{r-1}}-e_{n-1})})\\
&=\phi(\overline{(s-me_n)\wedge
(e_{j_1}-e_{n-1})\wedge\cdots\wedge(e_{j_{r-1}}-e_{n-1})})\ (\text{modulo}\ P)\\
&=(e_{j_1}-e_{n-1})\wedge\cdots\wedge (e_{j_{r-1}}-e_{n-1})\\
&=\sigma(f_{j_1}\wedge\cdots\wedge f_{j_{r-1}})\\
&=\sigma\phi(\overline{v_{m,1}\wedge f_{j_1}\wedge\cdots\wedge f_{j_{r-1}}}).
\end{align*}
\begin{enumerate}[\text{Case} 2:]
\item $j_{r-1}=n-1.$
\end{enumerate}
We have
\begin{align*}
&\phi(\sigma(\overline{v_{m,1}\wedge f_{j_1}\wedge\cdots\wedge f_{n-1}}))\\
&=\phi(\overline{(s-me_{n-1})\wedge (e_{j_1}-e_{n-1})\wedge\cdots\wedge(e_{n}-e_{n-1})})\\
&=\phi(\overline{(s-me_n+me_n-me_{n-1})\wedge(e_{j_1}-e_{n-1})\wedge\cdots\wedge(e_n-e_{n-1})})\\
&=\phi(\overline{(s-me_n)\wedge(e_{j_1}-e_{n-1})\wedge\cdots\wedge(e_{n}-e_{n-1})})\\
&=\phi(-\overline{(s-me_n)\wedge(e_{j_1}-e_{n})\wedge\cdots\wedge(e_{n-1}-e_{n})})\\
&=-\phi(\overline{(s-me_n)\wedge f_{j_1}\wedge\cdots\wedge f_{n-1}})\\
&=-f_{j_1}\wedge\cdots\wedge f_{n-1}\\
&=\sigma(f_{j_1}\wedge\cdots\wedge f_{n-1})\\
&=\sigma\phi(\overline{v_{m,1}\wedge f_{j_1}\wedge\cdots\wedge f_{n-1}}).
\end{align*}
So $\phi$ is an $\F\sym{n-m}$-isomorphism. We thus get that $N\sim S^{(n-m-r,1^r)}+S^{(n-m-r+1,1^{r-1})}$
as $\bigwedge^{r-1}S^{(n-m-1,1)}\cong S^{(n-m-r+1,1^{r-1})}$.

When $r=n-m$, by Lemma \ref{L;p>2,calculation} and a similar discussion as above, we obtain $(\bigwedge^rS^{(n-1,1)})^{\sym{m}}=\langle v_{m,1}\wedge f_{m+1}\wedge\cdots\wedge f_{n-1}\rangle$ and $\mathcal{F}_m(\bigwedge^rS^{(n-1,1)})\cong sgn$. When $r>n-m$, by Lemma \ref{L;p>2,calculation}, $\mathcal{F}_m(\bigwedge^rS^{(n-1,1)})=0$ as $(\bigwedge^rS^{(n-1,1)})^{\sym{m}}=0$. The proof is now complete.
\end{proof}
\begin{lem}\label{L;p=2,calculation}
Let $p=2$, $a,b\in \mathbb{N}$ and $a\leq b<n$. For any $s\in\{b+1,\ldots,n\}$, we have
\begin{align*}
\sum_{(i_1,\ldots,i_a)\in J(1,a,b)}(e_{i_1}+e_s)\wedge\cdots\wedge(e_{i_a}+e_s)=\begin{cases}
v_{b,a}+bf_s, &\text{if}\ a=1,\\
v_{b,a}+(b-a+1)f_s\wedge v_{b,a-1}, &\text{if}\ a>1.
\end{cases}
\end{align*}
\end{lem}
\begin{proof}
The case $a=1$ is trivial by a direct calculation. We now consider the case $a>1$. For any $(i_1,\ldots,i_\ell,\ldots,i_a)\in J(1,a,b)$, we shall write $f_{i_1}\wedge\cdots\wedge\widehat{f_{i_\ell}}\wedge\cdots\wedge f_{i_a}$ to denote the corresponding vector of $\bigwedge^{a-1}S^{(n-1,1)}$ without the component $f_{i_\ell}$. Since $p=2$, we have
\begin{align*}
&\sum_{\substack{(i_1,\ldots,i_a)\in J(1,a,b)}}\!(e_{i_1}+e_s)\wedge\cdots\wedge(e_{i_a}+e_s)\\
&=\sum_{\substack{(i_1,\ldots,i_a)\in J(1,a,b)}}\!\!\!\!\!\!\!\!(f_{i_1}+f_s)\wedge\cdots\wedge(f_{i_a}+f_s)\\
&=v_{b,a}+f_s\wedge(\!\!\!\!\!\!\!\!\!\sum_{\substack{(i_1,\ldots,i_a)\in J(1,a,b)}}\!\sum_{c=1}^a
f_{i_1}\wedge\cdots\wedge\widehat{f_{i_c}}\wedge\cdots\wedge f_{i_a})\\
&=v_{b,a}+(b-a+1)f_s\wedge(\sum_{\underline{\mathbf{j}}\in J(1,a-1,b)}
\!\!\!f_{\underline{\mathbf{j}}})\\
&=v_{b,a}+(b-a+1)f_s\wedge v_{b,a-1},
\end{align*}
as desired.
\end{proof}
\begin{lem}\label{L;p=2,calculation1}
Let $p=2$ and $s=r+m-n+1$. If $r>1$, for any $v\in(\bigwedge^rS^{(n-1,1)})^{\sym{m}}$, we have
\[v=\begin{cases} \sum_{i=0}^r\overline{B_{i}^v}, & \text{if}\ r<n-m,\ r\leq m,\\
\sum_{i=0}^m\overline{B_{i}^v}, & \text{if}\ m<r<n-m,\\
\sum_{i=s}^r\overline{B_{i}^v}, & \text{if}\ n-m\leq r\leq m,\\
\sum_{i=s}^m\overline{B_{i}^v}, & \text{if}\ n-m\leq r,\ r>m,\end{cases}\]
where, for all $0<i<r$, $\overline{B_{i}^v}=
\sum_{\underline{\mathbf{i}}\in J^+(m+1,r-i)}k_{\underline{\mathbf{i}}}v_{m,i}\wedge f_{\underline{\mathbf{i}}}$ and $k_{\underline{\mathbf{i}}}\in\F$ for all $\underline{\mathbf{i}}\in J^+(m+1,r-i)$. Furthermore, $\overline{B_{r}^v}\in\langle v_{m,r}\rangle$.
\end{lem}
\begin{proof}
Recall that $v=\sum_{i=0}^r\overline{B_{i}^v}$. Moreover, we have $\overline{B_{i}^v}\in(\bigwedge^rS^{(n-1,1)})^{\sym{m}}$ for all $0\leq i\leq r$ as $v\in(\bigwedge^rS^{(n-1,1)})^{\sym{m}}$. If $B_i\neq \varnothing$, an easy calculation gives that
\begin{equation*}
\overline{B_{i}^v}\in\begin{cases} \langle v_{m,i}\wedge w_{i}\rangle, & \text{if}\ 0<i<r,\\
 \langle v_{m,r}\rangle, & \text{if}\ i=r,\end{cases}
\end{equation*}
where $w_i\in\langle\{f_{\underline{\mathbf{i}}}:\ \underline{\mathbf{i}}\in J^+(m+1,r-i)\}\rangle$.
When $m<r$, by $(4.1)$, note that $B_{i}=\varnothing$ for all $m<i\leq r$. When $n-m\leq r$, notice that $0<s$. Therefore, by $(4.1)$ again, $B_{i}=\varnothing$ for all $0\leq i<s$. We now get the corresponding results according to these facts and the given conditions. The lemma follows.
\end{proof}
\begin{lem}\label{L;F(S),p=2}
Let $p=2$ and $s=r+m-n+1$. If $r>1$, then we have
\[\mathcal{F}_m(S^{(n-r,1^r)})\sim\begin{cases} \sum_{i=0}^rS^{(n-m-r+i,1^{r-i})}, & \text{if}\ r<n-m,\ r\leq m,\\
\sum_{i=0}^mS^{(n-m-r+i,1^{r-i})}, & \text{if}\ m<r<n-m,\\
\sum_{i=s}^rS^{(n-m-r+i,1^{r-i})}, & \text{if}\ n-m\leq r\leq m,\\
\sum_{i=s}^mS^{(n-m-r+i,1^{r-i})}, & \text{if}\ n-m\leq r,\ r>m.\end{cases}\]
\end{lem}
\begin{proof}
It is sufficient to work on $\mathcal{F}_m(\bigwedge^rS^{(n-1,1)})$. Let $\ell\in \mathbb{N}$ and $\ell<r$. When $J^+(m+1,r)$, $J^+(m+1,r-\ell)\neq\varnothing$ and $\ell\leq m$, define
\begin{align*}
\mathcal{C}_{0}=\{f_{\underline{\mathbf{i}}}:\ \underline{\mathbf{i}}\in J^+(m+1,r)\},\
\mathcal{C}_{\ell}=\{v_{m,\ell}\wedge f_{\underline{\mathbf{j}}}:\ \underline{\mathbf{j}}\in J^+(m+1,r-\ell)\}.
\end{align*}
Set $\mathcal{C}_{\ell}=\varnothing$ if $J^+(m+1,r-\ell)=\varnothing$ or $m<\ell$ and put $\mathcal{C}_{i,j}=\biguplus_{k=i}^j\mathcal{C}_{k}$ for all $1\leq i\leq j<r$. When $r<n-m$ and $r\leq m$, define $\mathcal{C}=\{v_{m,r}\}\cup\mathcal{C}_{0}\cup\mathcal{C}_{1,r-1}$. Note that each vector of $\mathcal{C}$ is fixed by $\sym{m}$. Moreover, $\mathcal{C}$ is linearly independent over $\F$. By Lemma \ref{L;p=2,calculation1}, for any $v\in(\bigwedge^rS^{(n-1,1)})^{\sym{m}}$, $v=\sum_{a=0}^r\overline{B_{a}^v}$, $\overline{B_{0}^v}\in\langle \mathcal{C}_0\rangle$, $\overline{B_{r}^v}\in\langle v_{m,r}\rangle$, $ \overline{B_{a}^v}=\sum_{\underline{\mathbf{j}}\in J^+(m+1,r-a)}k_{\underline{\mathbf{j}}}v_{m,a}\wedge f_{\underline{\mathbf{j}}},$
where $1\leq a<r$ and $k_{\underline{\mathbf{j}}}\in\F$ for all $\underline{\mathbf{j}}\in J^+(m+1,r-a)$. By these facts, we conclude that $\mathcal{C}$ is a basis of $(\bigwedge^rS^{(n-1,1)})^{\sym{m}}$.

We shall construct the desired Specht filtration for $\mathcal{F}_m(\bigwedge^rS^{(n-1,1)})$. Let $U_0=\langle \mathcal{C}_{0}\rangle$ and $U_i=\langle
\mathcal{C}_{0}\cup\mathcal{C}_{1,i}\rangle$ for all $1\leq i<r$. Set $U_r=\mathcal{F}_m(\bigwedge^rS^{(n-1,1)})$. We have a chain of subspaces of
$\mathcal{F}_m(\bigwedge^rS^{(n-1,1)})$ given by
\begin{equation*}
0\subset U_0\subset U_1\subset\cdots\subset U_r.
\end{equation*}
By the definition of $U_0$, note that $U_0\cong S^{(n-m-r,1^r)}$ as $\F\sym{n-m}$-modules. We claim that $U_i$ is an $\F\sym{n-m}$-module for all $1\leq i<r$. According to the definition of $U_i$, it is enough to check that every element of $\mathcal{C}_{1,i}$ is still in $U_i$ under the action of the transposition $(n-1,n)$. Write $\sigma=(n-1,n)$. For any $v_{m,\ell}\wedge f_{\underline{\mathbf{j}}}\in\mathcal{C}_{1,i}$, note that $1\leq\ell\leq i<r\leq m$. If $\ell>1$, we have the following calculation:
\begin{align*}
\sigma(v_{m,\ell}\wedge f_{\underline{\mathbf{j}}})=\sigma v_{m,\ell}\wedge \sigma f_{\underline{\mathbf{j}}}
=(v_{m,\ell}+(m-\ell+1)f_{n-1}\wedge v_{m,\ell-1})\wedge \sigma f_{\underline{\mathbf{j}}}\in U_i,
\end{align*}
where the last equality is from Lemma \ref{L;p=2,calculation}. Similarly, one can finish the checking if $\ell=1$. The claim is shown. Finally, we will prove that $U_i/U_{i-1}\cong \bigwedge^{r-i}S^{(n-m-1,1)}$ as $\F\sym{n-m}$-modules for all $1\leq i\leq r$. For any $v\in U_i$, write $\overline{v}$ to denote the image of $v$ under the natural map from $U_i$ to $U_i/U_{i-1}$. When $1\leq i<r$, define a linear map $\phi_i$ from $U_i/U_{i-1}$ to $\bigwedge^{r-i}S^{(n-m-1,1)}$ by sending each $\overline{v_{m,i}
\wedge f_{\underline{\mathbf{j}}}}$ to $f_{\underline{\mathbf{j}}}$. It is easy to observe that $\phi_i$ is a linear isomorphism. To show that $\phi_i$ is an $\F\sym{n-m}$-isomorphism, we only need to check that the action of $\sigma$ is preserved by $\phi_i$. If $i=1$, one may develop a similar checking as the one given in Lemma \ref{L;F(S),p>2}. If $1<i<r$, one can use the calculation of checking that $U_i$ is an $\F\sym{n-m}$-submodule to finish the checking. Therefore, $\phi_i$ is an $\F\sym{n-m}$-isomorphism for all $1\leq i<r$. Define a linear map $\phi_r$ from $U_r/U_{r-1}$ to $\F$ by sending $\overline{v_{m,r}}$ to a generator of $\F$. By Lemma \ref{L;p=2,calculation}, it is obvious to see that $\phi_r$ is an $\F\sym{n-m}$-isomorphism. We now have shown that $\mathcal{F}_m(\bigwedge^rS^{(n-1,1)})$ has the desired Specht filtration as $\bigwedge^{r-i}S^{(n-m-1,1)}\cong S^{(n-m-r+i,1^{r-i})}$ for all $1\leq i\leq r$.

The proofs of all the other cases are similar to the one of the case $r<n-m$ and $r\leq m$. We only list corresponding assertions for each case. One may follow the proof presented above to justify them.
When $m<r<n-m$, $(\bigwedge^rS^{(n-1,1)})^{\sym{m}}$ has a basis $\mathcal{C}_{0}\cup\mathcal{C}_{1,m}$. Let $U_0=\langle \mathcal{C}_{0}\rangle$, $U_i=\langle\mathcal{C}_{0}\cup\mathcal{C}_{1,i}\rangle$ for all $1\leq i<m$ and $U_m=\mathcal{F}_m(\bigwedge^rS^{(n-1,1)})$. Then the chain of the $\F\sym{n-m}$-modules
$$ 0\subset U_0\subset U_1\subset\cdots\subset U_m$$
gives the desired Specht filtration to $\mathcal{F}_m(\bigwedge^rS^{(n-1,1)})$. When $n-m\leq r\leq m<n-1$, $(\bigwedge^rS^{(n-1,1)})^{\sym{m}}$ has a basis $\{v_{m,r}\}\cup\mathcal{C}_{s,r-1}$.
Let $U_i=\langle\mathcal{C}_{s,i}\rangle$ for all $s\leq i<r$ and
$U_{r}=\mathcal{F}_m(\bigwedge^rS^{(n-1,1)})$. Then the chain of the $\F\sym{n-m}$-modules
$$ 0\subset U_s\subset U_{s+1}\subset\cdots\subset U_r$$
provides the desired Specht filtration for $\mathcal{F}_m(\bigwedge^rS^{(n-1,1)})$. When $m=n-1$, $(\bigwedge^rS^{(n-1,1)})^{\sym{m}}$ has a basis $\{v_{m,r}\}$. So $\mathcal{F}_m(\bigwedge^rS^{(n-1,1)})\cong S^{(1)}$. When $n-m\leq r$ and $r>m$, $(\bigwedge^rS^{(n-1,1)})^{\sym{m}}$ has a basis $\mathcal{C}_{s,m}$.
Set $U_i=\langle\mathcal{C}_{s,i}\rangle$ for all $s\leq i<m$ and $U_m=\mathcal{F}_m(\bigwedge^rS^{(n-1,1)})$. Then the chain of the $\F\sym{n-m}$-modules
$$ 0\subset U_s\subset U_{s+1}\subset\cdots\subset U_m$$
induces the desired Specht filtration of $\mathcal{F}_m(\bigwedge^rS^{(n-1,1)})$. The proof is now complete.
\end{proof}
A direct corollary of Lemmas \ref{L;S(n-1,1)}, \ref{L;F(S),p>2} and \ref{L;F(S),p=2} is given as follows.

\begin{cor}\label{C;Hook filtration}
Let $a$ and $b$ be integers such that $0\leq a<n$ and $1\leq b\leq n$. Then the following assertions hold.
\begin{enumerate}
\item [\em(i)] If $\mathcal{F}_b(S^{(n-a,1^a)})\neq 0$, then there exists a Specht filtration of $\mathcal{F}_b(S^{(n-a,1^a)})$ such that all the successive quotient factors of the filtration are labelled by the hook partitions of $n-b$.
\item [\em(ii)] Let $M$ be an $\F\sym{n}$-module with a Specht filtration and $\mathcal{F}_b(M)\neq 0$. If $b<p$ and all the successive quotient factors of the filtration are labelled by hook partitions of $n$, then $\mathcal{F}_b(M)$ has a Specht filtration with all successive quotient factors labelled by hook partitions of $n-b$.
\end{enumerate}
\end{cor}
\begin{proof}
Note that $\mathcal{F}_n$ takes an $\F\sym{n}$-module to zero or a direct sum of trivial modules. Also notice that $\mathcal{F}_b(S^{(n)})$ is a trivial module. (i) is clear by Lemmas \ref{L;S(n-1,1)}, \ref{L;F(S),p>2} and \ref{L;F(S),p=2}. For (ii), as $\mathcal{F}_b$ is now exact, we get the desired result by (i) and induction on the number of quotient factors of the filtration.
\end{proof}

One more result is needed to prove Theorem \ref{T;A}. It is a part of \cite[Theorem 4.5]{Lim2}. To state it and finish the proof of Theorem \ref{T;A}, we need another notation. Let $a\in \mathbb{N}$ and $ ap\leq n$. Let $E_a$ be the elementary abelian $p$-subgroup $\langle \bigcup_{i=a}^b\{s_i\}\rangle$ of $\sym{n}$ where $b=\lfloor\frac{n}{p}\rfloor$ and $s_i=((i-1)p+1,(i-1)p+2,\ldots, ip).$
\begin{lem}\label{L;Hook sjt}
Let $p\mid n$ and $a$ be the remainder of $r$ when $r$ is divided by $p$. If $2\nmid a$, then the stable generic Jordan type of $S^{(n-r,1^r)}{\downarrow_{E_1}}$ is $[p-1]^b$, where $b\in \mathbb{N}$.
\end{lem}

We now begin to finish the proof of Theorem \ref{T;A}.
\begin{prop}\label{P;Hook n>p}
Let $n>p>2$. Then $S^{(n-r,1^r)}$ is a trivial source module if and only if $(n-r,1^r)\in JM(n)_p$.
\end{prop}
\begin{proof}
One direction is clear by Theorem \ref{T;Hemmer}. For the other direction, we may assume further that $p\mid n$, since it is well-known that $S^{(n-r,1^r)}$ is simple if $p\nmid n$ (see \cite[Theorem 24.1]{GJ1}). Let $a$ be the remainder of $r$ when $r$ is divided by $p$. By Lemmas \ref{L;Hook sjt} and \ref{T;permutaiton}, we may assume further that $r>1$ and $2\mid a$. We consider about the following cases.
\begin{enumerate}[\text{Case} 1:]
\item $a=0.$
\end{enumerate}
When $n-r\geq 2p$, by Lemma \ref{L;F(S),p>2}, we have
\begin{align*}
&\mathcal{F}_p(S^{(n-r,1^r)})\cong S^{(n-p-r,1^r)}\oplus S^{(n-p-r+1,1^{r-1})},\\
&\mathcal{F}_p(S^{(n-p-r+1,1^{r-1})})\cong
S^{(n-2p-r+1,1^{r-1})}\oplus S^{(n-2p-r+2,1^{r-2})}.
\end{align*}
Using Lemma \ref{L;Hook sjt}, we get that the stable generic Jordan type of $S^{(n-2p-r+2,1^{r-2})}{\downarrow_{E_3}}$ is $[p-1]^b$, where $b\in \mathbb{N}$. Therefore, $S^{(n-2p-r+2,1^{r-2})}$ is not a trivial source module by Lemma \ref{T;permutaiton}. By Corollary \ref{L;permutation}, the fact implies that $S^{(n-r,1^r)}$ is not a trivial source module. When $n-r<2p$, note that $n-r=p$ since $p\mid n-r$ and $n>r$. By Lemma \ref{L;p-permutation modules} (iii), we know that $S^{(r+1,1^{p-1})}$ is a trivial source module if and only if $S^{(p,1^r)}$ is a trivial source module. Notice that $r>p-1$ as $r>1$ and $a=0$. We now apply $\mathcal{F}_p$ to $S^{(r+1,1^{p-1})}$ and obtain that
$$\mathcal{F}_p(S^{(r+1,1^{p-1})})\cong S^{(r-p+1,1^{p-1})}\oplus S^{(r-p+2,1^{p-2})}$$
by Lemma \ref{L;F(S),p>2}. By Lemma \ref{L;Hook sjt}, the stable generic Jordan type of $S^{(r-p+2,1^{p-2})}{\downarrow_{E_2}}$ is $[p-1]^c$, where $c\in \mathbb{N}$. Therefore, by the same reason mentioned in the case $n-r\geq 2p$ and Lemma \ref{L;p-permutation modules} (iii), we deduce that $S^{(p,1^r)}$ is not a trivial source module.
\begin{enumerate}[\text{Case} 2:]
\item $0<a<p.$
\end{enumerate}
When $n-r>p$, by Lemma \ref{L;F(S),p>2}, we have
$$\mathcal{F}_p(S^{(n-r,1^r)})\cong S^{(n-p-r,1^r)}\oplus S^{(n-p-r+1,1^{r-1})}.$$
By Lemma \ref{L;Hook sjt} again, the stable generic Jordan type of $S^{(n-p-r+1,1^{r-1})}{\downarrow_{E_2}}$ is $[p-1]^d$, where $d\in \mathbb{N}$. So $S^{(n-p-r+1, 1^{r-1})}$ is not a trivial source module by Lemma \ref{T;permutaiton}. We thus deduce that $S^{(n-r,1^r)}$ is not a trivial source module by Corollary \ref{L;permutation}. For the left case, note that $n-r\neq p$ as $p\mid n$ and $0<a<p$. When $n-r<p$, if $a=p-1$, we have $n-r=1$ as $p\mid n$ and $n-r<p$. So $S^{(n-r,1^r)}$ is in fact the simple module $S^{(1^{n})}$. We thus assume that $a<p-1$ in the following discussion. By Lemma \ref{L;p-permutation modules} (iii), we know that $S^{(r+1,1^{n-r-1})}$ is a trivial source module if and only if $S^{(n-r,1^r)}$ is a trivial source module. Note that $r>p$. Otherwise, we get $2p\leq n-r+r<p+p=2p$, which is absurd. Moreover, observe that $n-r-1\equiv p-1-a\pmod p$, $2\mid p-1-a$ and $p-1-a>0$. One may write a similar proof as the one given in the case $n-r>p$ to deduce that $S^{(r+1,1^{n-r-1})}$ is not a trivial source module. Using Lemma \ref{L;p-permutation modules} (iii) again, we get that
$S^{(n-r,1^r)}$ is not a trivial source module.

By the analysis of above two cases, we have shown that $S^{(n-r,1^r)}$ is not a trivial source module if $n>p>2$, $p\mid n$ and $n-r>1$. Therefore, if $n>p>2$, $S^{(n-r,1^r)}$ is a trivial source module only if it is simple. The proof is now complete.
\end{proof}
\begin{rem}\label{R;Othercases}
Let $p>2$. If $n<p$, all the Specht modules are trivial source modules. Moreover, they are all simple. If $n=p$, let $a$ be an integer such that $0\leq a<p$. It is known that $S^{(p-a,1^a)}$ is a trivial source module if and only if $2\mid a$ (See \cite{Erdmann}). These trivial source Specht modules are non-simple except $S^{(p)}$ and $S^{(1^p)}$.
\end{rem}
We now deal with the case $p=2$. In this case, all the indecomposable Specht modules labelled by hook partitions of $n$ were classified by Murphy in \cite{Murphy}. Let us state her result as follows. If $2\nmid n$ and $n\geq 2r$, $S^{(n-r,1^r)}$ is indecomposable if and only if $n\equiv 2r+1\pmod {2^L}$, where $L\in \mathbb{N}$ and $2^{L-1}\leq r<2^L$. If $2\nmid n$ and $2r>n$, $S^{(n-r,1^r)}$ is indecomposable if and only if $n\equiv 2r+1\pmod {2^{L'}}$, where $L'\in \mathbb{N}$ and $2^{{L'}-1}\leq n-r-1<2^{L'}$. If $2\mid n$, all Specht modules labelled by hook partitions of $n$ are indecomposable. We need the following lemma to conclude the case.
\begin{lem}\label{L;p=2Hook}
Let $p=2$ and $2\nmid n$. If $S^{(n-r,1^r)}$ is indecomposable, then
\begin{enumerate}
\item [\em(i)] $S^{(n-r,1^r)}$ is a trivial source module.
\item [\em(ii)] $S^{(n-r,1^r)}$ is non-simple if and only if $r>1$ and $r\neq n-2,\ n-1$.
\end{enumerate}
\end{lem}
\begin{proof}
Recall that $r\geq 1$. We now work on $\bigwedge^rS^{(n-1,1)}$. As $2\nmid n$, it is well-known that $\bigwedge^rS^{(n-1,1)}\mid \bigwedge^r M^{(n-1,1)}$. (i) thus follows as $\bigwedge^r M^{(n-1,1)}$ is a permutation module. (ii) is a part of \cite[Theorem 23.7]{GJ1}.
\end{proof}
\begin{rem}\label{R;p=2Hook}
Let $p=2$. If $2\nmid n$, all the non-simple trivial source Specht modules labelled by hook partitions of $n$ are clear by Lemma \ref{L;p=2Hook} and the Murphy's result. If $2\mid n$, let $P$ be a Sylow $2$-subgroup of $\sym{n}$. By \cite[Corollary 4.4, Theorem 4.5]{MP}, note that $P$ is a vertex of $S^{(n-r,1^r)}$ and $S^{(n-r,1^r)}{\downarrow_P}$ is a $P$-source of $S^{(n-r,1^r)}$. In particular, only the trivial module is a trivial source Specht module in the case.
\end{rem}

It is trivial to see that $S^{(1)}$ is a trivial source module. Theorem \ref{T;A} is now proved by combining Theorem \ref{T;Hemmer}, Proposition \ref{P;Hook n>p} and Remarks \ref{R;Othercases}, \ref{R;p=2Hook}.

\section{The Specht modules labelled by two-part partitions}
The proof of Theorem \ref{T;C} will be presented in this section. To finish the proof, we first work on Specht modules labelled by partitions with at most two columns and then translate the obtained results to the Specht modules labelled by two-part partitions by Lemma \ref{L;p-permutation modules} (iii).

Let $\lambda=(\lambda_1,\ldots,\lambda_\ell)\vdash n$, $s\in \mathbb{N}$ and $s\leq \ell$. For our purpose, define $\overline{\lambda}^s$ to be the partition obtained from $\lambda$ by deleting the first $s$ rows of $\lambda$. Namely, we have $\overline{\lambda}^s=(\lambda_{s+1},\ldots, \lambda_\ell)\vdash n-\sum_{i=1}^s\lambda_i$ if $s<\ell$ and get $\varnothing$ if $s=\ell$. For any non-negative integer $m$, write $\ell_p(m)$ to denote the smallest non-negative integer $a$ such that $m<p^a$. We need some preliminary results.

\begin{lem}\label{L;James}
Let $\lambda=(\lambda_1,\ldots,\lambda_\ell)\vdash n$. We have
\begin{enumerate}
\item [\em(i)] \cite[Theorem 24.4]{GJ1} $\textnormal{Hom}_{\F\sym{n}}(\F,S^{\lambda})\neq 0$ if and only if $\lambda_i\equiv -1 \pmod {p^{\ell_p(\lambda_{i+1})}}$ for all $1\leq i<\ell$.
\item [\em(ii)] \cite[Corollary 13.17]{GJ1} $\textnormal{Hom}_{\F\sym{n}} (\F,S_{\lambda})\neq 0$ implies that $\lambda=(n)$ if $p>2$.
\end{enumerate}
\end{lem}
Let $\lambda\vdash n$. We will call $\lambda$ a James partition if
$\textnormal{Hom}_{\F\sym{n}}(\F,S^{\lambda})\neq 0$.
\begin{lem}\label{L;Jamestrivial}
Let $p>2$ and $\lambda$ be a James partition of $n$. Then $S^\lambda$ is a trivial source module if and only if $\lambda=(n)$.
\end{lem}
\begin{proof}
One direction is clear. For the other direction, let $P$ be a vertex of $S^{\lambda}$ and suppose that $\lambda\neq (n)$. As $S^{\lambda}$ is a trivial source module, we have $S^{\lambda}\mid \F_P{\uparrow^{\sym{n}}}$. Note that Lemma \ref{L;James} (ii) implies that $S^{\lambda}\not\cong \textnormal{Sc}_{\sym{n}}(P)$. However, since $\lambda$ is a James partition, we come to a contradiction as we get two trivial submodules for the transitive permutation module  $\F_P{\uparrow^{\sym{n}}}$. So $\lambda$ has to be $(n)$ if $S^{\lambda}$ is a trivial source module. The lemma is shown.
\end{proof}

Let $a$ and $b$ be non-negative integers such that $b<p-1$ and $n=a(p-1)+b$. If $p>2$, it is well-known that $S^{((a+1)^b, a^{p-1-b})}$ has the simple head $sgn$ (see \cite[Example 24.5 (iii)]{GJ1}). We now prove a result which may have its own interest.

\begin{prop}
Let $p>2$ and $n=a(p-1)+b$ where $a$ and $b$ are non-negative integers such that $b<p-1$. Then $S^{((a+1)^b, a^{p-1-b})}$ is a trivial source module if and only if $n<p$.
\end{prop}
\begin{proof}
One direction is trivial. To prove the other direction, write $\epsilon$ to denote $((a+1)^b, a^{p-1-b})$. If  $S^{\epsilon}$ is a trivial source module, note that $S_{{\epsilon^{'}}}$ is also a trivial source module by Lemma \ref{L;p-permutation modules} (ii). Moreover, $S_{{\epsilon}^{'}}$ has a trivial quotient module as $S^{\epsilon}$ has the simple head $sgn$. Let $P$ be a vertex of $S_{\epsilon^{'}}$. The two facts imply that $S_{\epsilon^{'}}\mid {\F_P}{\uparrow^{\sym{n}}}$ and $S_{\epsilon^{'}}\cong Sc_{\sym{n}}(P)$. Therefore, $S_{\epsilon^{'}}$ has a trivial submodule by the definition of $Sc_{\sym{n}}(P)$. Using Lemma \ref{L;James} (ii), we deduce that $\epsilon=(1^n)$ and $n<p$.
\end{proof}

For any $m\in \mathbb{N}$, define $\nu_p(m)$ to be the largest non-negative integer $a$ such that $p^a\mid m$. In particular, $\nu_p(m)\geq 0$. By convention, put $\nu_p(0)=\infty$. The following theorem is known as the Carter criterion. One may refer to \cite[Theorem 7.3.23]{GJ3} for a proof.
\begin{thm}\cite[Carter criterion]{GJ3} \label{T;Carter}
Let $\lambda$ be a $p$-restricted partition of $n$. Then $S^{\lambda}$ is simple if and only if $\nu_p(h_{a,b}^\lambda)=\nu_p(h_{a,c}^\lambda)$ for any two nodes $(a,b)$ and $(a,c)$ of $\lambda$.
\end{thm}

To finish the preparation, we need a result found by Hemmer in \cite{Hemmer1}.

\begin{thm}\cite[Theorem 4.5]{Hemmer1}\label{T;Hemmer1}
Let $\lambda=(\lambda_1,\ldots,\lambda_{\ell})\vdash n$ and $\lambda_1<p$. Then $\mathcal{F}_{\lambda_1}(S^{\lambda})\cong S^{\overline{\lambda}^1}.$
\end{thm}

We now begin to finish the proof of Theorem \ref{T;C}. Until the end of the section, fix $p>2$ and an integer $r$ satisfying $0\leq 2r\leq n$. Note that the partition $(2^r, 1^{n-2r})$ is a $p$-restricted partition of $n$. The proof of Theorem \ref{T;C} will be based on a sequence of lemmas.

\begin{lem}\label{L;nu_p}
Let $\lambda=(2^r, 1^{n-2r})\vdash n$ and $\lambda\notin JM(n)_p$. Let $a$ be the largest positive integer such that the nodes $(a,1)$, $(a,2)$ are in $\lambda$ and $\nu_p(h_{a,1}^\lambda)\neq \nu_p(h_{a,2}^\lambda)$. Then there exist non-negative integers $u$, $v$ and $w$ such that $0\leq u,v<p$ and $h_{a,2}^{\lambda}=u+vp^w$.
\end{lem}
\begin{proof}
Note that $h_{a,2}^\lambda>0$ and let $h_{a,2}^\lambda$ have $p$-adic sum $\sum_{i\geq 0}b_ip^i$ where $0\leq b_i<p$ for all $i\geq 0$. If $h_{a,2}^\lambda=b_0$, we may assign $u=b_0$ and $v=w=0$. If $b_0<h_{a,2}^\lambda$, let $s$ be the smallest subscript such that $s>0$ and $b_s>0$. We claim that $\sum_{i>s}b_ip^i=0$. Suppose that  $\sum_{i>s}b_ip^i>0$. Set $t=a+\sum_{i>s}b_ip^i$ and note that the nodes $(t,1)$, $(t,2)$ are in $\lambda$. Moreover, we have $h_{t,1}^\lambda=h_{a,1}^\lambda- \sum_{i>s}b_ip^i$ and $h_{t,2}^\lambda=b_0+b_sp^s$. If $b_0\neq 0$, $\nu_p(h_{a,2}^\lambda)=\nu_p(h_{t,2}^\lambda)=0$. However, as $\nu_p(h_{a,1}^\lambda)\neq \nu_{p}(h_{a,2}^\lambda)$, observe that both $\nu_p(h_{a,1}^\lambda)$ and $\nu_p(h_{t,1}^\lambda)$ are non-zero. So $\nu_p(h_{t,1}^\lambda)\neq\nu_p(h_{t,2}^\lambda)$. It contradicts with the choice of $a$ as $t>a$. If $b_0=0$, $\nu_p(h_{a,2}^\lambda)=\nu_p(h_{t,2}^\lambda)=s$. By the choice of $a$, we conclude that $\nu_p(h_{t,1}^\lambda)=s$. It implies that $\nu_p(h_{a,1}^\lambda)=s$. We thus get $\nu_p(h_{a,1}^\lambda)=\nu_p(h_{a,2}^\lambda)=s$, which also contradicts with the choice of $a$. Therefore, the claim is shown. We can set $u=b_0$, $v=b_s$ and $w=s$ by the claim. The lemma follows.
\end{proof}
\begin{lem}\label{L;Case1}
Let $\lambda=(2^r, 1^{n-2r})\vdash n$ and $\lambda\notin JM(n)_p$. If $h_{1,2}^\lambda\geq p$, $\nu_p(h_{1,1}^\lambda)\neq\nu_p(h_{1,2}^\lambda)$ and $\nu_p(h_{a,1}^\lambda)=\nu_p(h_{a,2}^\lambda)$ for all nodes $(a,1)$, $(a,2)$ of $\lambda$ such that $a>1$, then $\nu_p(h_{1,1}^\lambda)$, $\nu_p(h_{1,2}^\lambda)>0$.
\end{lem}
\begin{proof}
By the hypotheses and Lemma \ref{L;nu_p}, let $h_{1,1}^\lambda$ and $h_{1,2}^\lambda$ have $p$-adic sums $\sum_{i\geq 0}a_ip^i$ and $b_0+b_sp^s$ respectively, where $0\leq a_i, b_0,b_s<p$ for all $i\geq 0$ and $b_s,s>0$. Suppose that either $\nu_p(h_{1,1}^\lambda)$ or $\nu_p(h_{1,2}^\lambda)$ is zero. Without loss of generality, we may assume that $\nu_p(h_{1,1}^\lambda)=0$. So we have $a_0\neq 0$ while $b_0=0$. Note that $h_{1,1}^\lambda>a_0$ otherwise $\lambda\in JM(n)_p$ as $\lambda$ is a $p$-core. Let $u=a_0+1$ and observe that the nodes $(u,1)$ and $(u,2)$ are in $\lambda$. Moreover, $h_{u,1}^\lambda=\sum_{i>0}a_ip^i$ and $h_{u,2}^\lambda=b_sp^s-a_0$. In particular, from the hypotheses, we have $\nu_p(h_{u,1}^\lambda)=\nu_p(h_{u,2}^\lambda)$. But it is clear to see that $\nu_p(h_{u,1}^\lambda)\geq 1$ and $\nu_p(h_{u,2}^\lambda)=0$. It shows that $\nu_p(h_{u,1}^\lambda)\neq\nu_p(h_{u,2}^\lambda)$, which is a contradiction. One may write a similar proof to get a contradiction for the case $\nu_p(h_{1,2}^\lambda)=0$. Therefore, we conclude that $\nu_p(h_{1,1}^\lambda)$, $\nu_p(h_{1,2}^\lambda)>0$. The proof is now finished.
\end{proof}

\begin{lem}\label{L;Case2}
Let $\lambda=(2^r, 1^{n-2r})\vdash n$ and $\lambda\notin JM(n)_p$. If $\nu_p(h_{1,1}^\lambda)\neq\nu_p(h_{1,2}^\lambda)$ and $\nu_p(h_{a,1}^\lambda)=\nu_p(h_{a,2}^\lambda)$ for all nodes $(a,1)$, $(a,2)$ of $\lambda$ such that $a>1$, then $S^\lambda$ is not a trivial source module.
\end{lem}
\begin{proof}
By the hypotheses and Lemmas \ref{L;nu_p}, \ref{L;Case1}, if $h_{1,2}^\lambda\geq p$, let $h_{1,1}^\lambda$ and $h_{1,2}^\lambda$ have $p$-adic sums $\sum_{i>0}a_ip^i$ and $b_sp^s$ respectively, where $0\leq a_i,b_s<p$ for all $i>0$ and $b_s,s>0$. Let $t$ be the smallest subscript such that $t>0$ and $a_t>0$. Note that $\nu_p(h_{1,1}^\lambda)=t\neq s=\nu_p(h_{1,2}^\lambda)$ by computation and the hypotheses. We claim that $s<t$. Suppose that $s>t$. Let $u=1+a_tp^t$ and note that the nodes $(u,1)$, $(u,2)$ are in $\lambda$. Moreover, $h_{u,1}^\lambda=\sum_{i\geq t+1}a_ip^i$ and $h_{u,2}^\lambda=b_sp^s-a_tp^t$. In particular, from the hypotheses, $\nu_p(h_{u,1}^\lambda)=\nu_p(h_{u,2}^\lambda)$. However, we also get that $\nu_p(h_{u,1}^\lambda)=\nu_p(\sum_{i\geq t+1}a_ip^i)\geq t+1$ but
$\nu_p(h_{u,2}^\lambda)=\nu_p(p^t(b_sp^{s-t}-a_t))=t$.
The calculation implies that $\nu_p(h_{u,1}^\lambda)\neq\nu_p(h_{u,2}^\lambda)$, which is a contradiction. The claim is now shown.

Suppose that $S^\lambda$ is a trivial source module. Notice that $\lambda'=(h_{1,1}^\lambda-1, h_{1,2}^\lambda)$. By Lemma \ref{L;p-permutation modules} (iii), we know that $S^{\lambda'}$ is a trivial source module. If $h_{1,2}^\lambda<p$, $\nu_p(h_{1,1}^\lambda)>0$ as $\nu_p(h_{1,1}^\lambda)\neq \nu_p(h_{1,2}^\lambda)=0$. Note that $\lambda'$ is a James partition by Lemma \ref{L;James} (i). Due to Lemma \ref{L;Jamestrivial} , we get $h_{1,2}^\lambda=0$. It is a contradiction since $\lambda\notin JM(n)_p$. If $h_{1,2}^\lambda\geq p$, by the claim and Lemmas \ref{L;Case1}, \ref{L;James} (i), we observe that $\lambda'$ is also a James partition as $s<t$. Therefore, by Lemma \ref{L;Jamestrivial}, we get $\lambda'=(n)$. It also contradicts with the condition that $\lambda\notin JM(n)_p$.  Hence $S^{\lambda}$ is not a trivial source module. The proof is now complete.
\end{proof}
\begin{lem}\label{L;Case3}
Let $\lambda=(2^r, 1^{n-2r})\vdash n$ and $\lambda\notin JM(n)_p$. Then $S^\lambda$ is not a trivial source module.
\end{lem}
\begin{proof}
According to the assumption that $\lambda\notin JM(n)_p$ and Theorem \ref{T;Carter}, there exists a pair of nodes $(a,1)$, $(a,2)$ of $\lambda$ such that $\nu_p(h_{a,1}^\lambda)\neq\nu_p(h_{a,2}^\lambda)$ and $\nu_p(h_{b,1}^\lambda)=\nu_p(h_{b,2}^\lambda)$ for all nodes $(b,1)$, $(b,2)$ of $\lambda$ such that $b>a$.
When $a=1$, the desired result follows by Lemma \ref{L;Case2}. When $a>1$, put $u=a-1$ and $v=2a-2$. By Theorem \ref{T;Hemmer1} and Corollary \ref{L;permutation}, $S^\lambda$ is a trivial source $\F\sym{n}$-module only if $S^{\overline{\lambda}^u}$ is a trivial source $\F\sym{n-v}$-module. Note that $\overline{\lambda}^u$ satisfies all the hypotheses of Lemma \ref{L;Case2}. Therefore, we deduce that $S^{\overline{\lambda}^u}$ is not a trivial source module, which shows that $S^\lambda$ is not a trivial source module. The proof is now complete.
\end{proof}
\begin{cor}\label{C;two-column}
Let $\lambda=(2^r, 1^{n-2r})\vdash n$. Then $S^\lambda$ is a trivial source module if and only if $\lambda\in JM(n)_p$.
\end{cor}
\begin{proof}
One direction is clear by Theorem \ref{T;Hemmer} while Lemma \ref{L;Case3} implies the other direction. The corollary follows.
\end{proof}

We close the section by presenting the proof of Theorem \ref{T;C}.
\begin{proof}[Proof of Theorem \ref{T;C}]
One direction is from Theorem \ref{T;Hemmer}. Note that $S^\lambda$ is simple if and only if $S^{\lambda'}$ is simple. Therefore, the other direction can be justified by using Lemma \ref{L;p-permutation modules} (iii) and Corollary \ref{C;two-column}. We are done.
\end{proof}
\section{Proof of Theorem \ref{T;B}}
The section provides the proof of Theorem \ref{T;B}. Theorem \ref{T;B} will be deduced from Theorem \ref{Con;Hudson} after a necessary preparation. Unlike the former sections, our main tool used here is the Brou\'{e} correspondence of trivial source $\F\sym{n}$-modules.

We state a stronger result than the conjecture of Hudson as follows and prove it after a preparation. The result implies her conjecture obviously.
\begin{thm}\label{Con;Hudson}
Let $p=2$, $n\geq 4$ and $\lambda\vdash n$ with $2$-weight $2$. If $S^\lambda$ is an indecomposable, non-simple, trivial source module, then $S^\lambda\cong Y^\mu$ where $\mu=\kappa_2(\lambda)+(2^2)$. 
\end{thm}
We list some lemmas to finish the preparation.
\begin{lem}\label{L;outertensor}
Let $G$, $H$ be finite groups and $M$, $N$ be an $\F G$-module and an $\F H$-module respectively. If both $M$ and $N$ are self-dual, then $M\boxtimes N$ is also self-dual as an $\F[G\times H]$-module.
\end{lem}
\begin{proof}
Let $\B_M=\{m_1,\ldots, m_s\}$ be a basis of $M$ and $\B_N=\{n_1,\ldots,n_t\}$ be a basis of $N$. For any $\phi\in M^*$, $\psi\in N^*$ and $v=\sum_{i=1}^s\sum_{j=1}^tk_{i,j}m_i\boxtimes n_j\in M\boxtimes N$, view $\phi\boxtimes \psi\in (M\boxtimes N)^*$ by setting $(\phi\boxtimes\psi)
(v)=\sum_{i=1}^s\sum_{j=1}^tk_{i,j}\phi(m_i)\psi(n_j)$. Note that $M^*\boxtimes N^*=(M\boxtimes N)^*$ under the viewpoint. So we only need to show that $M\boxtimes N\cong M^*\boxtimes N^*$ as $\F[G\times H]$-modules. It is clear since $M\cong M^*$ and $N\cong N^*$. 
\end{proof}
\begin{lem}\label{L;Greenself-dual}
Let $G$ be a finite group and $M$ be an indecomposable $\F G$-module with a vertex $P$. If $N_G(P)\leq H\leq G$, then $M$ is self-dual if and only if $\mathcal{G}_H(M)$ is self-dual.
\end{lem}
\begin{proof}
By properties of $\mathcal{G}_H(M)$, we have
\begin{align}
M{\downarrow_H}\cong \mathcal{G}_H(M)\oplus U,\ \mathcal{G}_H(M){\uparrow^G}\cong M\oplus V,
\end{align}
where, for any indecomposable direct summand $W$ of the $\F H$-module $U$ or the $\F G$-module $V$, $P$ is not a vertex of $W$. When $M\cong M^*$, by (6.1), we get
\begin{align}
\mathcal{G}_H(M)^*\oplus U^*\cong (M{\downarrow_H})^*=(M^*){\downarrow_H}\cong M{\downarrow_H}\cong \mathcal{G}_H(M)\oplus U.
\end{align}
Note that $*$ preserves vertices of indecomposable modules. We thus deduce that $\mathcal{G}_H(M)\cong \mathcal{G}_H(M)^*$ by (6.2) and the Krull-Schmidt Theorem. If $\mathcal{G}_H(M)\cong \mathcal{G}_H(M)^*$, by (6.1) again, we obtain
\begin{align}
M^*\oplus V^*\cong (\mathcal{G}_H(M){\uparrow^G})^*\cong (\mathcal{G}_H(M)^*){\uparrow^G}\cong \mathcal{G}_H(M){\uparrow^G}\cong M\oplus V.
\end{align}
By (6.3) and the Krull-Schmidt Theorem, we have $M\cong M^*$. The lemma follows.
\end{proof}
\begin{lem}\label{L;p=2simple}
Let $p=2$ and $\lambda$ be a $2$-regular or $2$-restricted partition of $n$. If $S^\lambda$ is self-dual, then $S^\lambda$ is simple.
\end{lem}
\begin{proof}
Suppose that $S^\lambda$ is not simple. When $\lambda$ is $2$-regular, $S^\lambda$ has a simple head $D^\lambda$. As $S^\lambda$ is self-dual, it forces that $S^\lambda$ has at least two composition factors isomorphic to $D^\lambda$. It is an obvious contradiction. When $\lambda$ is $2$-restricted, note that $S^\lambda\cong S^{\lambda'}$ as $S^\lambda$ is self-dual. We also get a contradiction similarly. The lemma follows.
\end{proof}
The following lemma is a straightforward fact on symmetric groups.
\begin{lem}\label{L;Normalizer1}
Let $H\leq \sym{n}$ and $r$ be an integer such that $0<r\leq n$. Let $H$ act on $\mathbf{n}$. If $H$ fixes $n-r+1,\ldots, n$, then $$N_{\sym{n}}(H)=N_{\sym{n-r}}(H)\times \sym{r}\leq \sym{(n-r,r)}.$$ Moreover, $N_{\sym{n}}(H)/H\cong(N_{\sym{n-r}}(H)/H)\times \sym{r}$.
\end{lem}
\begin{lem}\cite[Corollary 2]{Wei}\label{L;Normalizer2}
Let $P$ be a Sylow $2$-subgroup of $\sym{n}$. Then we have $N_{\sym{n}}(P)=P$.
\end{lem}
We now deal with Theorem \ref{Con;Hudson}. From now on, write $C_4$, $C_2\times C_2$ and $K_4$ to denote the subgroups $\langle (1,2,3,4)\rangle$, $\langle (1,2),(3,4)\rangle$ and $\langle (1,2)(3,4), (1,3)(2,4)\rangle$ of $\sym{4}$ respectively. Note that they are exactly all $2$-subgroups of $\sym{4}$ of order $4$ up to $\sym{4}$-conjugation.
\begin{lem}\label{P;defectgroup Block}
Let $p=2$ and $M$ be an indecomposable $\F\sym{n}$-module. Let $\lambda$ be a $2$-core and $w$ be the $2$-weight of $B_\lambda$. If $B_\lambda$ contains $M$ and $M$ is a trivial source module whose vertices are defect groups of $B_\lambda$, then $M\cong Y^{\alpha}\cong S^{\alpha}$, where $\alpha=\lambda+(2w)\in JM(n)_2$.
\end{lem}
\begin{proof}
Let $P$ be a Sylow $2$-subgroup of $\sym{2w}$. Note that $P$ is also a vertex of $M$. By Lemmas \ref{L;Normalizer1} and \ref{L;Normalizer2}, we know that $N_{\sym{n}}(P)=N_{\sym{2w}}(P)\times \sym{n-2w}=P\times \sym{n-2w}$. Moreover, $N_{\sym{n}}(P)/P\cong\sym{n-2w}$. Since $M$ is an indecomposable module with a vertex $P$ and a trivial $P$-source, by Theorem \ref{T;Broue} (i), we get that $M(P)$ is an indecomposable projective $\F\sym{n-2w}$-module. In particular, $M(P)\cong Y^\mu$ as $\F\sym{n-2w}$-modules, where $\mu$ is a $2$-restricted partition of $n-2w$. Let $\nu=\mu+(2w)$. By Theorem \ref{T; Youngvertices}, note that $Y^{\nu}$ is an indecomposable $\F\sym{n}$-module with a vertex $P$ and a trivial $P$-source. By the Sylow Theorem, we can require that $P$ is a Sylow $2$-subgroup of $\sym{\O_\nu}$. According to Theorem \ref{Erdmann} and Lemma \ref{L;Normalizer2}, $\sym{n-2w}\cong \N_{\sym{n}}(P)/\N_{\sym{\O_{\nu}}}(P)=N_{\sym{n}}(P)/P$ and $Y^{\nu}(P)\cong Y^{\mu}$. We thus obtain that $M\cong Y^{\nu}$ by Theorem \ref{T;Broue} (i). It forces that $\alpha=\nu$ as $B_\lambda$ contains $M$. Note that there does not exist $\beta\vdash n$ such that $\kappa_2(\beta)=\lambda$ and $\beta\rhd \alpha$. So $Y^{\alpha}\cong S^{\alpha}$ by \cite[2.6]{Donkin}. Moreover, $S^{\alpha}$ is simple by Lemma \ref{L;p=2simple}.
\end{proof}
\begin{lem}\label{L;C2timeC_2}
Let $p=2$, $n\geq 4$ and $M$ be an indecomposable $\F\sym{n}$-module. Let $\lambda$ be a $2$-core and $B_\lambda$ have $2$-weight $2$. If $M$ is a trivial source module with a vertex $C_2\times C_2$ and $B_\lambda$ contains $M$, then $M\cong Y^{\alpha}$, where $\alpha=\lambda+(2^2)$.
\end{lem}
\begin{proof}
Notice that $M$ is isomorphic to a Young module since $M\mid M^{(2^2,1^{n-4})}$. As $C_2\times C_2$ is a vertex of $M$, by Theorem \ref{T; Youngvertices}, there exists a $2$-restricted partition of $n-4$, say $\mu$, such that $M\cong Y^\nu$, where $\nu=\mu+(2^2)$. Since $B_\lambda$ contains $M$ and $B_\lambda$ has $2$-weight $2$, we get that $\alpha=\nu$. The lemma is shown.
\end{proof}
\begin{lem}\label{L;decompositon}
Let $p=2$. Then $\F_{K_4}{\uparrow^{\sym{4}}}\cong 2D^{(3,1)}\oplus Sc_{\sym{4}}(K_4)$. In particular, all the indecomposable direct summands of the decomposition of $\F_{K_4}{\uparrow^{\sym{4}}}$ are self-dual.
\end{lem}
\begin{proof}
Note that $N_{\sym{4}}(K_4)/K_4\cong \sym{3}$. By an easy calculation,
$\F_{K_4}{\uparrow^{\sym{4}}}(K_4)\cong\F\sym{3}$ as $\F\sym{3}$-modules. It is well-known that $\F\sym{3}\cong Y^{(1^3)}\oplus 2Y^{(2,1)}$ and $D^{(3,1)}\mid \F_{K_4}{\uparrow^{\sym{4}}}$. As both $D^{(3,1)}$ and $Sc_{\sym{4}}(K_4)$ have the vertex $K_4$, by Lemma \ref{L;Brauer} (i), Theorem \ref{T;Broue} (i), (iii) and counting dimensions, $D^{(3,1)}(K_4)\cong Y^{(2,1)}$ and $(Sc_{\sym{4}}(K_4))(K_4)\cong Y^{(1^3)}$. The desired isomorphic formula follows by Theorem \ref{T;Broue} (iii). The second assertion is clear by the isomorphic formula.
\end{proof}
\begin{prop}\label{L;K4}
Let $p=2$, $n\geq 4$ and $H=N_{\sym{n}}(K_4)$. Let $M$ be an indecomposable $\F\sym{n}$-module with a vertex $K_4$. If $M$ is a trivial source module, then, for some $2$-restricted partition $\lambda$ of $n-4$, we have $\mathcal{G}_H(M)\cong D^{(3,1)}\boxtimes Y^\lambda$ or $Sc_{\sym{4}}(K_4)\boxtimes Y^{\lambda}$. In particular, $M$ is self-dual.
\end{prop}
\begin{proof}
By Lemma \ref{L;Normalizer1}, we have $H=N_{\sym{4}}(K_4)\times \sym{n-4}=\sym{4}\times \sym{n-4}$. Moreover, $H/K_4\cong \sym{(3,n-4)}$. By Theorem \ref{T;Broue} (i), note that $M(K_4)\cong Y^{(2,1)}\boxtimes Y^{\lambda}$ or $Y^{(1^3)}\boxtimes Y^{\lambda}$ as $\F\sym{(3,n-4)}$-modules, where $\lambda$ is a $2$-restricted partition of $n-4$. Furthermore, from \cite[Propositions 1.1, 1.2]{KL}, we know that $D^{(3,1)}\boxtimes Y^\lambda$ and $Sc_{\sym{4}}(K_4)\boxtimes Y^\lambda$ are indecomposable $\F\sym{(4,n-4)}$-modules with a vertex $K_4$ and a trivial $K_4$-source. By Lemma \ref{L;Brauer} (ii), we also have
\begin{align}
&(D^{(3,1)}\boxtimes Y^\lambda)(K_4)\cong D^{(3,1)}(K_4)\boxtimes Y^\lambda\cong Y^{(2,1)}\boxtimes Y^\lambda,\\
& (Sc_{\sym{4}}(K_4)\boxtimes Y^\lambda)(K_4)\cong (Sc_{\sym{4}}(K_4))(K_4)\boxtimes Y^\lambda\cong Y^{(1^3)}\boxtimes Y^\lambda
\end{align}
as $\F\sym{(3,n-4)}$-modules. Note that $N_{\sym{(4,n-4)}}(K_4)=\sym{4}\times\sym{n-4}=H$. Therefore, $(6.4)$ and $(6.5)$ are isomorphic formulae for $\F[H/K_4]$-modules. From Theorem \ref{T;Broue} (ii), we get $\mathcal{G}_H(M)\cong \mathcal{G}_H(D^{(3,1)}\boxtimes Y^\lambda)$ or $\mathcal{G}_H(Sc_{\sym{4}}(K_4)\boxtimes Y^\lambda)$. However, observe that $\mathcal{G}_H(D^{(3,1)}\boxtimes Y^\lambda)=D^{(3,1)}\boxtimes Y^\lambda$ and $\mathcal{G}_H(Sc_{\sym{4}}(K_4)\boxtimes Y^\lambda)=Sc_{\sym{4}}(K_4)\boxtimes Y^\lambda$. The first assertion thus follows. The second one is clear by Lemmas \ref{L;decompositon}, \ref{L;outertensor} and \ref{L;Greenself-dual}.
\end{proof}


\begin{proof}[Proof of Theorem \ref{Con;Hudson}]
Let $P$ be a vertex of $S^\lambda$. We may choose $P$ to be a $2$-subgroup of $\sym{4}$. According to the hypotheses and \cite[Theorem 1]{Wildon}, the order of $P$ is at least $4$ and $P$ is not cyclic. The cases where $P$ is a Sylow $2$-subgroup of $\sym{4}$ or a $2$-subgroup conjugated to $C_2\times C_2$ are done by Lemmas \ref{P;defectgroup Block} and \ref{L;C2timeC_2}.

To prove Theorem \ref{Con;Hudson}, it suffices to show that $P\neq K_4$. Suppose that $P=K_4$. By Proposition \ref{L;K4}, notice that $S^\lambda$ is self-dual. It implies that $\lambda$ is neither $2$-regular nor $2$-restricted by Lemma \ref{L;p=2simple}. As $\lambda$ has $2$-weight $2$, when $n>4$, it forces that $\lambda=(s+2,s-1,s-2,\ldots,2,1^3)$ for some positive integer $s$. From the fact and \cite[Theorem A]{GE}, we get that $P$ contains $C_2\times C_2$ up to $\sym{n}$-conjugation, which is a contradiction. When $n=4$, we have $\lambda=(2^2)$. It is well-known that $S^{(2^2)}\cong D^{(3,1)}$, which contradicts the assumption that $S^\lambda$ is non-simple. So $P\neq K_4$. The proof is now complete.
\end{proof}

Theorem \ref{T;B} is now shown by Theorems \ref{T;Hemmer} and \ref{Con;Hudson}. We end the paper with a corollary which may have its own interest.

\begin{cor}
Let $p=2$, $n\geq 4$ and $\lambda\vdash n$ with $2$-weight $2$. Then $S^\lambda$ is an indecomposable $\F\sym{n}$-module with a vertex $K_4$ and a trivial $K_4$-source if and only if $n=4$ and $\lambda=(2^2)$.
\end{cor}
\begin{proof}
One direction is clear. For the other direction, notice that $K_4$ can not be $\sym{n}$-conjugate to the Sylow $2$-subgroups of the Young subgroups of $\sym{n}$. By Theorems \ref{T; Youngvertices} and \ref{T;B}, we thus deduce that $\lambda\in JM(n)_2$. Due to \cite[Main Theorem]{GMathas}, it is sufficient to show that $\lambda$ is neither $2$-regular nor $2$-restricted. It is well-known that $S^\lambda\cong Y^\lambda$ if $\lambda\in JM(n)_2$ and $\lambda$ is $2$-regular. Similarly, $S^\lambda\cong Y^{\lambda'}$ if $\lambda\in JM(n)_2$ and $\lambda$ is $2$-restricted. Therefore, as $S^\lambda$ has a vertex $K_4$, $\lambda$ is neither $2$-regular nor $2$-restricted. The desired result thus follows.
\end{proof}

\subsection*{Acknowledgments}
The author thanks his supervisor Dr. Kay Jin Lim for some suggestions of refining this paper. He also gratefully thanks Professor David Hemmer for letting him know the question of classification of trivial source Specht modules via a private communication in the conference `Representation Theory of Symmetric Groups and Related Algebras' held in Singapore. Finally, he gratefully thanks an anonymous referee for his or her careful reading and helpful suggestions.

\end{document}